\newcommand{\dataversione}{August 27, 2022}

\documentclass[11pt,reqno,a4paper]{amsart}
\usepackage{amssymb}
\usepackage{fancyhdr}
\usepackage{mathrsfs}
\usepackage{hyperref}\hypersetup{colorlinks=true, citecolor=blue}
\usepackage{enumitem}

\numberwithin{equation}{section}

\newtheoremstyle{mytheorem}
{10pt}
{10pt}
{\it}
{\parindent}
{\bf}
{.}
{ }
{\thmnumber{#2.~}\thmname{#1}\thmnote{~\rm#3}}

\newtheoremstyle{myremark}
{10pt}
{10pt}
{\rm}
{\parindent}
{\bf}
{}
{ }
{\thmnumber{#2.~}\thmname{#1}\thmnote{~\rm#3}}

\newtheoremstyle{myparagraph}
{10pt}
{10pt}
{\rm}
{\parindent}
{\bf}
{.}
{ }
{\thmnumber{#2.~}\thmname{#1}\thmnote{#3}}

\theoremstyle{mytheorem}
\newtheorem{theorem}[subsection]{Theorem}

\newtheorem{lemma}[subsection]{Lemma}
\newtheorem{corollary}[subsection]{Corollary}
\newtheorem{proposition}[subsection]{Proposition}

\theoremstyle{myremark}
\newtheorem{remark}[subsection]{Remark.}
\newtheorem{remarks}[subsection]{Remarks.\hskip-13pt}

\theoremstyle{myparagraph}
\newtheorem{parag}[subsection]{}
\newtheorem*{parag*}{}

\makeatletter
    \def\@secnumfont{\sc}
    \def\section{\@startsection%
    {section}
    {1}
    \z@{1.5\linespacing\@plus .2\linespacing}
      {.7\linespacing}
      {\normalfont\sc\centering}}
\makeatother

\makeatletter
    \def\sottosezione{\@startsection{subsection}{2}
    \z@{1\linespacing\@plus .2\linespacing}
      {.6\linespacing}
      {\normalfont\bf}}
\makeatother

\makeatletter
\renewenvironment{proof}[1][\proofname]{\par 
    \pushQED{\qed}%
    \normalfont \topsep10\p@\@plus6\p@\relax 
    \trivlist 
    \item[\hskip\labelsep 
    \bfseries 
    #1\@addpunct{.}]\ignorespaces}
    {\popQED\endtrivlist\@endpefalse} 
\providecommand{\proofname}{Proof}
\makeatother

\newcommand{\footnoteb}[1]{\footnote{~#1}}

\newlist{enumeraterem}{enumerate}{1}
\setlist[enumeraterem]{label=(\roman*), leftmargin=0pt, itemsep=3pt, itemindent=30pt}
\newlist{enumeratethm}{enumerate}{1}
\setlist[enumeratethm]{label={\rm(\roman*)}, leftmargin=30pt, itemsep=2pt}

\newcommand{\bfi}{\mathbf{i}}
\newcommand{\bfj}{\mathbf{j}}
\newcommand{\R}{\mathbb{R}}
\newcommand{\HH}{\mathbb{H}}
\newcommand{\Mass}{\mathbb{M}}
\newcommand{\Haus}{\mathscr{H}}
\newcommand{\F}{\mathscr{F}}
\newcommand{\Leb}{\mathscr{L}}
\newcommand{\I}{\mathscr{I}}
\newcommand{\wrt}{with respect to\ }
\newcommand{\Gr}{\mathrm{Gr}}
\newcommand{\Span}{\mathrm{span}}
\newcommand{\sign}{\mathrm{sign}}
\newcommand{\supp}{\mathrm{spt}}
\newcommand{\loc}{\mathrm{loc}}
\newcommand{\de}{\mathrm{d}}
\newcommand{\bd}{\partial}
\newcommand{\hodge}{\mathop{\star}}
\newcommand{\Vhat}{\smash{\widehat V}}
\newcommand{\Vbar}{\smash{\overline V}}
\newcommand{\ic}[1]{[\kern-1pt[#1]\kern-1pt]}
\newcommand{\scalar}[2]{\langle #1;\, #2\rangle}
\newcommand{\bigscalar}[2]{\big\langle #1;\, #2\big\rangle}
\DeclareMathOperator{\trace}{\mbox{\Large$\llcorner$}}
\DeclareMathOperator{\antitrace}{\mbox{\kern-1.5pt\Large$\lrcorner$\kern.5pt}}
\DeclareMathOperator{\dive}{div}
\DeclareMathOperator{\esterno}{\mbox{\large$\wedge$}}
\newcommand{\passo}[1]{\medskip\textit{#1}}	

\begin{document}

\thispagestyle{empty}
~\vskip -1.1 cm

	%
	%
{\footnotesize\noindent 
[version: final, \dataversione]
\hfill \emph{J.~Differential~Geom.}, 122 (2022), n.~1, 1-33\par
\hfill DOI~\href{http://dx.doi.org/10.4310/jdg/1668186786}{10.4310/jdg/1668186786} \par
}

\vspace{1.7 cm}

	%
	%
{\centering\Large\bf
On the geometric structure of currents\\
\vskip3pt
tangent to smooth distributions\\
}

\vspace{.7 cm}

	%
	%
{\centering\sc 
Giovanni Alberti, Annalisa Massaccesi, Eugene Stepanov
\\
}

\vspace{.8 cm}

	%
	%
{\rightskip 1 cm
\leftskip 1 cm
\parindent 0 pt
\footnotesize
{\sc Abstract.}
It is well known that a $k$-dimensional smooth surface in a Euclidean 
space cannot be tangent to a non-involutive distribution of 
$k$-dimensional planes.
In this paper we discuss the extension of this statement 
to weaker notions of surfaces, namely integral and normal
currents. We find out that integral currents behave to this regard
exactly as smooth surfaces, while the behavior of normal currents
is rather multifaceted.
This issue is strictly related to a geometric property 
of the boundary of currents, which is also discussed in details.


\par
}

\section{Introduction}
\label{sec:intro}
The starting point of this paper is the following implication in 
Frobenius theorem: if $V$ is a distribution of $k$-dimensional 
planes on an open set $\Omega$ in $\R^n$, and $\Sigma$ 
is a $k$-dimensional smooth surface 
which is everywhere tangent to $V$, then $V$ is involutive at every point of
$\Sigma$ or, equivalently, $\Sigma$ does not intersect the 
open set where $V$ is non-involutive.
In the following we refer to this statement 
simply as Frobenius theorem.

In the classical statement it is assumed that  
both the distribution $V$ and the surface $\Sigma$ are 
sufficiently regular. In particular it suffices that 
$V$ be of class $C^1$ and $\Sigma$ be a surface (submanifold)
of class $C^1$, possibly with boundary.
In this paper we discuss the generalization of this result to 
weaker notions of surfaces, though
not weakening the regularity assumption on $V$ 
(see however~\S\ref{nonsmooth}).

We first remark that Frobenius theorem 
does not hold if $\Sigma$ is just
a closed subset of a $k$-dimensional  $C^1$-surface. 
More precisely, 
for every continuous distribution of $k$-planes $V$ there
exists a $C^1$-surface $S$ such that the set $\Sigma$ of all 
points of $S$ where $S$ is tangent to $V$ has positive
$k$-dimensional measure, 
regardless of the involutivity of $V$ 
(this result was proved for a special $V$ in \cite{Balogh}, 
Theorem~1.4; 
the general version can be found in~\cite{Alb-Mas-Mer}).

On the other hand, Frobenius theorem holds
if the boundary of $\Sigma$ (relative to $S$)
is not too large; for example, it suffices that 
the $(k-1)$-dimensional Hausdorff measure
$\Haus^{k-1}(\bd\Sigma)$ be finite, 
see~\S\ref{tangencysets} for more details.

However the most satisfactory version of this statement 
is obtained by considering surfaces and boundary in the sense of currents: 
in Theorem~\ref{thm:frobenius} we show that 
Frobenius theorem holds for all integral currents
tangent to the distribution $V$.
Thus one is naturally led to wonder what happens for the 
largest class of currents with ``nice'' boundary, 
namely normal currents. It turns out that this case
is much more interesting, and in particular the validity of 
Frobenius theorem depends also on how ``diffuse'' the current is
(Theorem~\ref{thm:diffuse}).

Notice that our results are local in nature, and therefore, 
even if stated in the Euclidean space, they actually hold 
in Riemannian manifolds, and even in Finsler manifolds.

Some of the results in this paper were announced in~\cite{Alb-Mas-Lincei}.

\sottosezione*{Description of the results}
Through the rest of this paper $V$ is a $C^1$-distribution
of $k$-planes on an open set $\Omega$ in $\R^n$, 
and $N(V)$ is the open set of all points where $V$ is 
non-involutive.%
\footnoteb{Most of the terminology used in this introduction
is properly defined in Section~\ref{sec:not};
the precise definition of $N(V)$ is given in~\S\ref{def:inv}.}

\begin{theorem}
\label{thm:frobenius}
Let $T$ be an integral 
current in $\Omega$ which is tangent to $V$.%
\footnoteb{The precise meaning of ``$T$ is tangent to $V$'' is
given in \S\ref{def:distrib}.}
Then the support of $T$ does not intersect the non-involutivity 
set~$N(V)$.
\end{theorem} 

A version of this statement was first proved in the second author's 
dissertation (\cite{Massaccesi-tesi}, Theorem~2.2.6), 
following a completely different argument.

\medskip
The next step is to consider normal currents.
We recall here that these currents share many properties
with integral currents, including that of having a ``nice''
boundary, but differ in many regards. 
In particular integral $k$-dimensional currents 
are supported on $k$-dimensional (rectifiable) sets, 
while $k$-dimensional normal currents can be 
quite ``sparse'', even absolutely continuous 
\wrt the Lebesgue measure.

\smallskip
The following example, proposed by 
M.~Zworski in~\cite{Zworski}, shows that
Frobenius theorem does not hold in general 
for normal currents.

\begin{parag}[Example]
\label{ex:nonfrob}
Consider a simple $k$-vectorfield $v=v_1\wedge\dots\wedge v_k$
of class $C^1$ on $\R^n$ and let $T$ be the $k$-current given by 
$T=v\,\Leb^n$. Then $T$ is a normal current on every bounded 
open set $\Omega$ in $\R^n$
(see \S\ref{def:curr} and Remark~\ref{rem:curr-iii}) 
and it is clearly tangent to the distribution $V$ spanned 
by $v_1, \dots, v_k$, regardless of its involutivity.
\end{parag}

It turns out that there is a general result
behind this example:  
a normal current $T$ which is tangent 
to a distribution $V$ 
must be sufficiently ``sparse'' on the 
non-involutivity set $N(V)$, and the degree of ``sparseness''
depends on how much non-involutive the distribution $V$~is.

\medskip
A precise statement requires some preparation.
We let $\Vhat$ be the distribution spanned by 
the vectorfields tangent to $V$ and their first commutators
(see \S\ref{def:vhat} for precise definitions),
and for every $d=k, \dots, n$ we set 
\[
N(V,d) := \big\{ x\in\Omega \colon \dim(\Vhat(x))=d\big\}
\, .
\]
(Thus $N(V)$ is the union of all sets $N(V,d)$ with $d=k+1, \dots, n$.)

\smallskip
We then consider a normal $k$-current $T$ on $\Omega$, 
which we write as $T=\tau\mu$ where $\mu$ 
is a finite positive measure and $\tau$ is 
a $k$-vectorfield which is nonzero $\mu$-a.e.\
(cf.~\S\ref{def:curr} and Remark~\ref{rem:curr-i});
in particular the support of $T$ agrees with the 
support of~$\mu$. 
Along the same line we write the boundary of $T$ as
$\bd T=\tau'\mu'$.

We say that $T$ is tangent to $V$ if
$\Span(\tau(x))=V(x)$ for $\mu$-a.e.~$x$.%
\footnoteb{The span of a (non necessarily simple) $k$-vector 
is defined in \S\ref{def:span}.}

We define the degree of sparseness of 
a measure in terms of absolute continuity \wrt
the Hausdorff measure $\Haus^d$ or the 
integral geometric measures $\I^d_t$
(the higher is $d$, the sparser is the measure).
If $T$ is a normal $k$-current and we take $\mu$ and $\mu'$ as above, 
then  $\mu$ is absolutely continuous \wrt $\I^k_t$ 
and therefore also \wrt $\Haus^k$ (see~\S\ref{def:intgeo}),
that is, $\mu \ll \I^k_t \ll \Haus^k$.
Similarly $\mu' \ll \I^{k-1}_t \ll \Haus^{k-1}$.

\medskip
We can now state the main result for normal currents.

\begin{theorem}
\label{thm:diffuse}
Take $V$, $T=\tau\mu$ and $\bd T=\tau'\mu'$ as above, and assume 
that $T$ be tangent to $V$. Then 
\begin{enumeratethm}
\item\label{thm:diffuse-i}
the restriction of $\mu$ to the set $N(V)$ satisfies
$\mu\trace N(V) \ll \mu'$;

\item\label{thm:diffuse-ii}
for $d>k$ there holds
$\mu\trace N(V,d) \ll \I^d_t \ll \Haus^d$;

\item\label{thm:diffuse-iii}
for $d>k$ there holds
$\mu'\trace N(V,d) \ll \Haus^{d-1}$.
\end{enumeratethm}
\end{theorem}

Using Theorem~\ref{thm:diffuse} we can show that Frobenius theorem
holds for normal currents that satisfy certain additional conditions:

\begin{corollary}
\label{cor:normalfrob}
Take $V$, $T=\tau\mu$ and $\bd T=\tau'\mu'$ as above, 
and assume that $T$ be tangent to~$V$.
If any of the following conditions holds then the support
of $T$ does not intersect~$N(V)$:
\begin{enumeratethm}[label=\rm(\alph*)]

\item\label{cor:normalfrob-a}
$\mu$ is concentrated on a $\I_t^{k+1}$-null Borel set;%
\,\footnoteb{We say that $\mu$ is concentrated on 
a Borel set $E$ if $\mu(\Omega\setminus E)=0$;
this implies that the support of $\mu$ is contained
in the closure of $E$ (but not necessarily in~$E$).}

\item\label{cor:normalfrob-b}
$\mu$ is concentrated on a $\mu'$-null Borel set;%

\item\label{cor:normalfrob-c}
$T$ is a rectifiable current (possibly with non-integral multiplicity);

\item\label{cor:normalfrob-d}
$\bd T=0$.
\end{enumeratethm}
\end{corollary}

\begin{remarks}\label{rem:normalfrob}
\begin{enumeraterem}[ref={\ref*{rem:normalfrob}(\roman*)}]
\item
\label{rem:normalfrob-i}
Regarding condition~(a) in Corollary~\ref{cor:normalfrob}, 
we recall that the following implications hold for every Borel
set $E$:
\[
\I_t^{k+1}(E)=0 
\ \Leftarrow\
\Haus^{k+1}(E)=0
\ \Leftarrow\
\dim_H(E)<k+1
\, ,
\]
where $\dim_H(E)$ is the Hausdorff dimension of $E$.

\item
\label{rem:normalfrob-ii}
Under condition~(c), Corollary~\ref{cor:normalfrob} 
generalizes Theorem~\ref{thm:frobenius}.

\item
\label{rem:normalfrob-iii}
Even though the measures $\mu$ and $\mu'$ in the representations 
of $T$ and $\bd T$ are not unique, the statements of Theorem~\ref{thm:diffuse} 
and Corollary~\ref{cor:normalfrob} do not depend on the choice 
of these measures (cf.\ Remark~\ref{rem:curr-ii}).

\end{enumeraterem}
\end{remarks}

As already pointed out in \cite{Alb-Mas-Lincei}, 
the validity of Frobenius theorem for normal currents 
is strictly related to the following property
of the boundary.

\begin{parag}[Geometric property of the boundary]
\label{def:gpb}
Let $T=\tau\mu$ be a normal $k$-current on the open set $\Omega$
with boundary $\bd T=\tau'\mu'$. 
We say that $T$ has the \emph{geometric property of the boundary} if, 
up to a modification of $\tau$ in a $\mu$-null set, 
\begin{enumeratethm}[label=\rm(\alph*)]
\item\label{thm:gpb-i}
the map $x\mapsto\Span(\tau(x))$
is continuous on the support of $T$;
\item\label{thm:gpb-ii}
$\Span(\tau'(x)) \subset \Span(\tau(x))$ for $\mu'$-a.e.~$x$.
\end{enumeratethm}
It is easy to check that if $T$ is tangent to the distribution 
$V$ then these conditions are equivalent to the inclusion
\begin{equation}
\label{e:gpb}
\Span(\tau'(x)) \subset V(x)
\quad\text{for $\mu'$-a.e.~$x$.}
\end{equation}
\end{parag}

\begin{theorem}
\label{thm:gpb-frob}
Take $V$ and $T=\tau\mu$ as above, 
and assume that $T$ be tangent to~$V$.
Then the following assertions are equivalent:
\begin{enumeratethm}
\item\label{thm:gpb-frob-i}
$T$ has the geometric property of the boundary, 
that is, \eqref{e:gpb} holds;
\item\label{thm:gpb-frob-ii}
the support of $\mu$ does not intersect~$N(V)$.
\end{enumeratethm}
\end{theorem}

\begin{remarks}\label{rem:gpb1}
\begin{enumeraterem}[ref={\ref*{rem:gpb1}(\roman*)}]
\item
The current associated to an oriented surface $\Sigma$ of class $C^1$ 
has the geometric property of the boundary; indeed condition 
\ref{thm:gpb-ii} in \S\ref{def:gpb} reduces to the fact that for 
every $x\in \bd\Sigma$ the tangent space $T_x(\bd\Sigma)$ is 
contained in~$T_x\Sigma$.

\item
Example~\ref{ex:nonfrob} and Theorem~\ref{thm:gpb-frob}
show that there are normal currents $T$ which are tangent to 
a distribution $V$ of class $C^1$ and do not have
the geometric property of the boundary.
In this case one may ask where the inclusion 
$\Span(\tau'(x)) \subset V(x)$ holds and
where it does not; a detailed answer is given 
in Theorem~\ref{thm:gpb-adv} and Remark~\ref{rem:gpb-adv}.

\item\label{rem:gpb1-iii}
Theorem~\ref{thm:gpb-frob} implies that
the geometric property of the boundary holds if
$T$ is tangent to a distribution of $k$-planes 
of class $C^1$ and satisfies one of the conditions 
\ref{cor:normalfrob-a}--\ref{cor:normalfrob-d} in 
Corollary~\ref{cor:normalfrob}, e.g., 
if $T$ is an integral current.
In \cite{Alb-Mas} we give an example of 
integral current which is tangent to a 
\emph{continuous} distribution of $k$-planes
and does not have the geometric property of the boundary.
\end{enumeraterem}
\end{remarks}

\sottosezione*{Additional comments}
\begin{parag}[On the geometric property of the boundary]
\label{rem:gpb2}
We collect here further remarks on the property defined in~\S\ref{def:gpb}.

\begin{enumeraterem}
\item
The continuity requirement~\ref{thm:gpb-i} in~\S\ref{def:gpb}.
is needed to make the definition meaningful.
Indeed if we drop this requirement then every current $T$
such that $\bd T$ is singular \wrt $T$ (that is, $\mu'$
is singular \wrt $\mu$) has the geometric
property of the boundary---the point is that the $k$-vectorfield 
$\tau$ is only determined up to $\mu$-null sets, 
and therefore it can be arbitrarily modified in 
a set which is $\mu'$-full.

\item
The relation between the geometric property of the boundary 
and Frobenius theorem for currents was first pointed out
by the second author in her dissertation~\cite{Massaccesi-tesi}, 
where a version of Theorem~\ref{thm:frobenius} is obtained
as a corollary of the geometric property of the boundary 
of integral currents (\cite{Massaccesi-tesi}, Lemma~2.2.1).

\item
We point out that in \cite{Massaccesi-tesi} 
and \cite{Alb-Mas} the sentence ``the current $T$ 
is tangent to the distribution $V$''
has a stronger meaning than in this paper.
Here it means that the tangent plane 
$\Span(\tau(x))$ is prescribed $\mu$-a.e., 
while there it means that both the tangent plane
$\Span(\tau)$ and its orientation are prescribed
$\mu$-a.e.

Under this stronger notion of tangency, in \cite{Alb-Mas} 
it is proved that the geometric property of the boundary 
holds for integral currents that are tangent to a continuous
distribution $V$ of $k$-planes (while here we need that
$V$ be of class~$C^1$, cf.~Remark~\ref{rem:gpb1-iii}).
\end{enumeraterem}
\end{parag}

\begin{parag}[An open problem]
\label{openproblem}
The statement of Theorem~\ref{thm:diffuse} depends 
crucially on the sets $N(V,d)$, which are defined  
using the distribution $\Vhat$ spanned by the vectorfields 
tangent to $V$ and their first commutators (see~\S\ref{def:vhat}).
In this context it is also natural to consider 
the distribution $\Vbar$
spanned by the Lie algebra generated by $V$, 
that is, by the vectorfields tangent to $V$
and their commutators of all orders.
Clearly the distribution $\Vbar$ contains $\Vhat$, 
and the inclusion may be strict.  
If this is the case, replacing the sets $N(V,d)$ by
\[
{\overline N}(V,d) := \big\{ x\in\Omega \colon \dim(\Vbar(x))=d \big\}
\]
in statements~\ref{thm:diffuse-ii} and \ref{thm:diffuse-iii} 
of Theorem~\ref{thm:diffuse} yields a stronger results. 
We believe that such results are true, 
but cannot be obtained by a modification 
of the present proof.
\end{parag}

\begin{parag}[Sobolev sets]
\label{sobolevsets}
Extensions of Frobenius theorem to weaker 
notions of surfaces have been studied by many authors.
For instance, in \cite{Ma-Ma-Mo}, Theorem~2.1, it 
is proved that \emph{Sobolev sets} of dimension
$m$ in the (sub-Riemannian) Heisenberg group $\HH^n$
cannot be horizontal for $m>n$, that is, images of Sobolev maps 
with derivative of rank $m$ from open subsets of $\R^m$ 
into $\R^{2n+1}\simeq\HH^n$ cannot be tangent to the 
horizontal distribution.

Using Theorem~\ref{thm:frobenius} we partially recover
this result and extend it to a more general setting:

\begin{theorem}
\label{thm:ntv}
Let $\Omega$ be an open subset in $\R^n$ and $V$ 
a $C^1$-distri\-bu\-tion of $k$-planes on $\Omega$.
Let $A$ be an open set in $\R^k$ and let 
$u:A\to\Omega$ be a continuous map of class 
$\smash{W^{1,p}_\loc}$ with $p>k$
such that, for a.e.~$z\in A$, the image of 
the differential of $u$ at $z$ is $V(u(z))$. 
Then $u(A)$ does not intersect the 
non-involutivity set~$N(V)$.
\end{theorem}
\end{parag}

\begin{parag}[Tangency sets]
\label{tangencysets}
Given a distribution of $k$-planes $V$ and 
a $k$-dimensional surface $S$ of class $C^1$, we say that a 
closed subset $\Sigma$ of $S$ is a \emph{tangency set}
of $S$ and $V$ if the tangent space $T_xS$ agrees with $V(x)$
for every $x\in\Sigma$.
In this context the statement of Frobenius theorem 
reduces to 
\begin{equation}
\label{e:tangencysets}
\Haus^k\big( \Sigma\cap N(V) \big) =0 
\, .
\end{equation}
As already pointed out at the beginning of this introduction, 
if $S$ is of class~$C^1$ then \eqref{e:tangencysets} does not hold 
for all tangency sets $\Sigma$. 
However it holds if $\Sigma$ has finite perimeter relative to $S$.%
\footnoteb
{This claim follows from two results by S.~Delladio: 
in \cite{Del12}, Corollary~4.1, he proves that 
$\Haus^k$-a.e.~point $x$ of a finite perimeter set $\Sigma$ 
is a \emph{superdensity} point, i.e.,  
$\Haus^k\big( (B(x,r)\cap S)\setminus\Sigma \big) =o(r^{k+1})$,  
and in \cite{Del_BV}, Corollary~1.1, 
he proves that the set of superdensity points of
$\Sigma$ does not intersect~$N(V)$, 
and therefore \eqref{e:tangencysets} holds.}
Note that this condition is implied by (but not equivalent to)
any of the following:
a)~the (topological) boundary of $\Sigma$ relative to $S$ is $\Haus^{k-1}$-finite;
b)~the boundary of the canonical current associated to $\Sigma$ has finite mass.

On the other hand, if the surface $S$ is of class $C^2$, 
or even of class $C^{1,1}$, then \eqref{e:tangencysets} holds 
for every tangency set $\Sigma$, regardless of the regularity of
its boundary (see for instance \cite{Ba_Pin_Roh}, Theorem~1.3).
This result is generalized in \cite{Alb-Mas-Mer} 
by proving that if $S$ is of class $C^{1,\alpha}$ 
for some $0<\alpha<1$ then Frobenius theorem holds
for every tangency set $\Sigma$ whose (distributional) 
boundary has a suitable \emph{fractional} regularity.
This shows that the validity of Frobenius theorem 
depends on a combination of the regularity of 
$\bd\Sigma$ and of the regularity of the surface~$S$. 

It would be interesting to extend this result to 
some class of currents.
\end{parag}

\begin{parag}[Carnot-Carath\'eodory spaces]
The tangency set $\Sigma$ with $\Haus^k(\Sigma)>0$
mentioned in \S\ref{tangencysets} is a non-trivial 
$k$-rectifiable set in~$\Omega$. 
However, if the distribution $V$ satisfies H\"ormander condition 
and we replace the Euclidean distance in $\Omega$ with the 
Carnot-Carath\'eodory distance associated to $V$, 
then $\Omega$ contains only trivial $k$-rectifiable sets,
and in particular $\Sigma$ is no longer rectifiable
(see \cite{Magnani-2} for the case of Heisenberg groups, 
and \cite{Alb-Mas-Mer-0} for a more general context).
\end{parag}

\begin{parag}[Non-smooth distributions]
\label{nonsmooth}
Through this paper we always assume that the distribution 
$V$ is of class $C^1$, which is the minimal regularity 
required to define involutivity in the classical sense
(see~\S\ref{def:inv}).
We show in \S\ref{def:weakinv} that it is possible 
to define involutivity also if $V$ is less regular than $C^1$
using a suitable distributional formulation.
In order to extend the results 
stated above to less regular~$V$, 
a major difficulty seems to be the correct 
definition of the non-involutivity set~$N(V)$.
\end{parag}

\begin{parag*}[Structure of the paper]
Section~\ref{sec:not} contains the notation and 
some preliminary results.
The main result in Section~\ref{sec:main} is the key 
identity \eqref{e:d}, which allows us to establish a very 
precise connection between Frobenius theorem for normal 
currents and the geometric property of the boundary 
(Theorem~\ref{thm:gpb-adv}).
All statements given in this introduction are more 
or less straightforward consequences of identity \eqref{e:d} and 
Theorem~\ref{thm:gpb-adv}; the proofs are collected 
in Section~\ref{sec:proofs}.  
\end{parag*}

{\small
\begin{parag*}[Acknowledgements]
This research was partly carried out during several visits of the authors:
A.M.\ at the Mathematics Department in Pisa (supported by the University of Pisa through 
the 2015 PRA Grant ``Variational methods for geometric problems'');
E.S.\ at the Mathematics Department in Pisa (supported by the 2018 INdAM-GNAMPA 
project ``Geometric Measure Theoretical approaches to Optimal Networks'');
G.A.\ and A.M.\ at CIRM in Trento (supported by the CIRM ``Research in Pairs'' program).

The research of G.A.\ has been partially supported by the Italian Ministry of University and Research (MIUR) 
through PRIN project 2010A2TFX2
and by the European Research Council (ERC) through project 291497. 
A.M.\ has been partially supported by ERC through project 306247 
and by the European Union's Horizon 2020 programme through project 752018.
E.S.~has been partially supported by the Russian Foundation for Basic Research (RFBR)
through grant {\#20-01-00630A}.

We thank Andrea Merlo for his comments on an earlier version of this paper
and for pointing out a mistake therein. 
\end{parag*}
}

\section{Notation and preliminary results}
\label{sec:not}
We assume that the reader is somewhat familiar with the theory of currents.
Therefore in this section we only briefly recall the basic notions 
of multilinear algebra and of the theory of currents, mainly to fix 
the notation, and describe in more details only those notions 
that are of less common use.

Through this paper we tacitly assume that sets and functions are 
Borel measurable and measures are defined on the Borel 
$\sigma$-algebra, and are real-valued and finite 
(with the notable exception of Lebesgue, Hausdorff 
and integral geometric measures).

Here is a list of frequently used notations:
\begin{itemize}
[leftmargin=50pt, itemsep=4 pt plus 1pt, labelsep=10pt]

\item[{$\mu\trace F$}]
restriction of a measure $\mu$ to a Borel set $F$, 
i.e., $[\mu\trace F](E) := \mu(E\cap F)$ for every Borel set $E$ in $X$;

\item[{$\rho\mu$}]
measure associated to a measure $\mu$ on $X$ and a Borel density $\rho$, 
that is, $[\rho\mu](E) := \int_E \rho \, d\mu$ for every Borel set $E$ in $X$;

\item[{$f_\#\mu$}]
pushforward of a measure $\mu$ on $X$ according 
to a Borel map $f:X\to Y$, that is, 
$[f_\#\mu](E) := \mu(f^{-1}(E))$ for every Borel set $E$ in~$Y$;

\item[{$f_\#T$}]
pushforward of a current $T$ according 
to a map~$f$ (see, e.g., \cite{Kra-Par}, \S7.4.2);

\item[{$|\mu|$}]
variation measure associated to a real- or vector-valued measure~$\mu$;

\item[{$\mu\ll\lambda$}]
the measure $\mu$ is absolutely continuous \wrt the measure
$\lambda$;

\item[{$\mu_a$, $\mu_s$}]
absolutely continuous and singular part of 
a measure $\mu$ \wrt a given measure~$\lambda$.

\item[{$\Leb^n, \ \Haus^d$}]
Lebesgue measure on $\R^n$ and $d$-dimensional Hausdorff measure;

\item[{$\I^d_t$}]
$d$-dimensional integral geometric measures (\S\ref{def:intgeo});

\item[$I(n,k)$]
set of all multi-indices $\bfi:=(i_1,\dots,i_k)$ 
with $1\le i_1<\dots<i_k\le n$;

\item[$\esterno_k(V)$]
space of $k$-vectors in a linear space $V$;  
the canonical basis of $\esterno_k(\R^n)$ is formed by the 
simple $k$-vectors $e_\bfi:=e_{i_1}\wedge\dots\wedge e_{i_k}$ 
with $\bfi\in I(n,k)$, where $\{e_1,\dots,e_n\}$
is the canonical basis of $\R^n$;
$\esterno_k(\R^n)$ is endowed with the Euclidean
norm $|\cdot|$ associated to this basis;%
\footnoteb{None of the results in this paper
depend on this choice of norm.}

\item[$\esterno^k(V)$]
space of $k$-covectors on a linear space $V$; 
the canonical basis of $\esterno^k(\R^n)$ is formed by the 
simple $k$-covectors 
$\de x_\bfi:=\de x_{i_1}\wedge\dots\wedge \de x_{i_k}$ 
with $\bfi\in I(n,k)$, where $\{\de x_1,\dots, \de x_n\}$
is the canonical basis of the dual of $\R^n$;
$\esterno^k(\R^n)$ is endowed with Euclidean norm $|\cdot|$
associated to this basis;

\item[$\de x$]
$:= \de x_1\wedge\dots\wedge \de x_n$;

\item[$\wedge$]
exterior product of $k$-vectors, or of $h$-covectors;

\item[$\antitrace \, , \, \trace$]
interior products of a $k$-vector and a $h$-covector
(\S\ref{def:int_prod}); 

\item[$\hodge$]
Hodge-type operator on $k$-vectors and $k$-covectors (\S\ref{def:hodge});

\item[$\Span(v)$]
span of a $k$-vector $v$ (\S\ref{def:span});

\item[$\de$]
exterior derivative of a $k$-form (\S\ref{def:diffop}); 

\item[$\dive$]
divergence of a $k$-vectorfield (\S\ref{def:diffop});

\item[{$[v,v']$}]
Lie bracket of vectorfields $v$ and $v'$ (\S\ref{def:lie});

\item[$W(\mu,\cdot)$]
decomposability bundle of a measure $\mu$ (\S\ref{def:decbun});

\item[{$N(V)$}]
non-involutivity set of a distribution of $k$-planes $V$ (\S\ref{def:inv});

\item[{$\Vhat$}]
and $N(V,d)$, see~\S\ref{def:vhat}.
\end{itemize}

\begin{parag}[Integral geometric measure]
\label{def:intgeo}
Given $d=1, \dots, n$ and $t \in[1,+\infty]$, 
we denote by $\I^d_t$ the $d$-dimensional integral geometric 
measure of exponent $t$ on $\R^n$
(the precise definition can be found in \cite{Federer}, \S2.10.5, 
or \cite{Kra-Par}, \S2.1.4).

The relevant features are that 
$\I^d_t$ is invariant under isometries of $\R^n$, it agrees with 
the Hausdorff measure $\Haus^d$ on regular 
$d$-dimensional surfaces of $\R^n$, and in general 
satisfies $\I^d_t \le \Haus^d$.
Moreover, and this is essential to this paper, a Borel 
set $E$ is $\I^d_t$-null if and only if
\[
\Haus^d \big( p_V(E) \big) = 0 
\quad
\text{for a.e.~$d$-plane $V$ in $\R^n$,}
\]
where $p_V$ stands for the orthogonal projection 
on $V$ and ``a.e.'' refers to the Haar measure on the 
Grassmannian of $d$-planes in $\R^n$.

Note that the class of $\I^d_t$-null Borel sets is the
same for all $t$ and is strictly larger than the class of
$\Haus^d$-null sets (indeed, by the Besicovitch-Federer 
projection theorem, the first class contains all sets which
are $\Haus^d$-finite and purely $d$-unrectifiable).

In particular the fact that a measure $\mu$ is absolutely continuous
\wrt~$\I^d_t$ does not depend on the exponent $t$
and implies that $\mu$ is absolutely continuous \wrt~$\Haus^d$
(but the converse does not hold).
\end{parag}

\sottosezione*{Multilinear algebra}
In this subsection we review the basic notions
of multilinear algebra; we consider multivectors and multicovectors 
in a general linear space~$V$. For a thorough treatise 
of this topic, we refer the reader to \cite{Federer}, \S1.5, 
from which we borrow the notation, or standard textbooks in 
Differential Geometry, such as \cite{Lee}, Chapters~11 and~14, 
and \cite{War}, Chapter~2.

\begin{parag}[Interior product]
\label{def:int_prod}
Given a $k$-vector $v$ in $V$ and an $h$-covector 
$\alpha$ on $V$ with $h\le k$, the \emph{interior product} 
$v\trace\alpha$ is the $(k-h)$-vector in $V$ defined by
\[
\scalar{v\trace\alpha}{\beta}:=\scalar{v}{\alpha\wedge\beta}
\quad\text{for every $(k-h)$-covector $\beta$;}
\]
if $k\le h$, the \emph{interior product} $v\antitrace\alpha$ 
is the $(h-k)$-covector defined by
\[
\scalar{w}{v\antitrace\alpha}:=\scalar{w\wedge v}{\alpha}
\quad\text{for every $(h-k)$-vector $w$.}
\]

Note that given a $k$-vector $v$, an $h$-covector $\alpha$ 
and an $h'$-covector $\alpha'$ with $h+h'\le k$, then
\[
v\trace(\alpha\wedge\alpha')=(v\trace\alpha)\trace\alpha'
\, .
\]
Similarly, given a $k$-vector $v$, a $k'$-vector $v'$ and an $h$-covector $\alpha$
with $k+k'\le h$, then
\[
(v\wedge v')\antitrace \alpha = v\antitrace(v'\antitrace\alpha)
\, .
\]
\end{parag}

\begin{parag}[Span of a $\boldsymbol k$-vector]
\label{def:span}
Given a $k$-vector $v$ in $V$, 
we denote by $\Span(v)$ the smallest of all linear subspaces 
$W$ of $V$ such that $v$ belongs to $\esterno_k(W)$.%
\footnoteb{For example, given linearly independent vectors $v_1,\dots,v_4$, 
then $\Span(v_1\wedge v_2 + v_3\wedge v_4)$ is the 
linear subspace spanned by $v_1,\dots,v_4$.}
This definition is well-posed
because every $k$-vector in $W$ is canonically identified
with a $k$-vector in $V$ via the inclusion map $i:W\to V$, 
and assuming this identification we have
$\esterno_k(W) \cap \esterno_k(W') = \esterno_k(W\cap W')$
for every $W,W'$ subspaces of $V$.
We have the following properties (see~\cite{Alb-Mar}, Proposition~5.9):
\begin{enumeratethm}
\item
if $v=0$ then $\Span(v)=\{0\}$;
\item
if $v\ne 0$ then $\dim(\Span(v)) \ge k$;
\item
if $v$ is simple and non-trivial, 
that is, $v=v_1\wedge\dots\wedge v_k$ with 
$v_1,\dots,v_k$ linearly independent vectors in $V$, 
then $\Span(v)$ is the subspace of $V$ spanned by 
$v_1,\dots, v_k$ and $\dim(\Span(v))=k$;
\item
\label{def:span-iv}
conversely, if $\dim(\Span(v))=k$ then 
$v$ is simple and non-trivial;
\item
\label{def:span-v}
$\Span(v)$ consists of all vectors of the form 
$v\trace\alpha$ with $\alpha\in\esterno^{k-1}(V)$.
\end{enumeratethm}
\end{parag}

The next lemma will be used later in the proofs.

\begin{lemma}
\label{lem:span0}
Let $v$ be a $k$-vector in $V$ and let $W$ be a $d$-dimensional 
subspace of~$V$.
Then $\Span(v)\subset W$ if any of the following conditions hold: 
\begin{enumeratethm}[label=\rm(\alph*)]
\item
\label{lem:span0-1}
there exist a subspace $W'$ of $W$ with $d':=\dim(W')\ge k$,
and an integer $h$ with $1 \le h\le d'-k+1$ such that 
$\Span(v\wedge w)\subset W$ for every $w\in\esterno^h(W')$;
\item
\label{lem:span0-2}
$v\wedge w=0$ for every $w\in\esterno^{d-k+1}(W)$;
\item
\label{lem:span0-3}
$k=1$ and there exist an integer $1\le h\le d$ and a simple $h$-vector $w\in\esterno^h(W)$
with $w\ne 0$ such that $\Span(v\wedge w) \subset W$;
\end{enumeratethm}
\end{lemma}

\begin{proof}
The proof is divided in three steps, one for each condition.

\passo{Step~1: if condition~\ref{lem:span0-1} holds then $\Span(v)\subset W$.}
We argue by contradiction, and prove that
if $\Span(v) \not\subset W$ then there exists 
$w\in\esterno^h(W')$ such that $\Span(v\wedge w)\not\subset W$.
To this aim we choose vectors $e_1,\dots,e_n$ in $V$ so that 
$e_1, \dots, e_{d'}$ form a basis of~$W'$,
$e_1,\dots, e_d$ form a basis of~$W$, 
and $e_1,\dots,e_n$ form a basis of~$V$. 
Then we write $v$ as
\[
v=\sum_{\bfi\in I(n,k)} \hskip -5 pt v_\bfi \, e_\bfi
\, .
\]
Since $\Span(v)$ is not contained in $W$ there exists a
multi-index $\bfj=(j_1, \dots, j_k)$ in $I(n,k)$ such that 
$v_\bfj\ne 0$ and $j_k>d$.
This means that at most $k-1$ indices in $\bfj$
belong to $\{1, \dots, d'\}$; thus there are at least 
$d'-k+1$ indices in $\{1, \dots, d'\}$ that are not in $\bfj$,
and since $h\le d'-k+1$ we can find 
a multi-index $\bfj'\in I(d',h)$ whose indices are all different 
from those of $\bf j$ 
(with a slight abuse of notation we write
$\bfj'\cap\bfj=\varnothing$).

We now set $w:=\hat e_{\bfj'}$. Then 
\[
v\wedge w =
\hskip - 5 pt \sum_{\bfi \colon \bfj'\cap\bfi=\varnothing} \hskip -10 pt
v_\bfi \, \hat e_\bfi \wedge \hat e_{\bfj'}
\, .
\]
Let $\bfj\cup\bfj'$ denote the multi-index in $I(n,k+h)$ that contains
the indices in~$\bfj$ and in~$\bfj'$. The formula above 
shows that the coordinate $(v\wedge w)_{\bfj\cup\bfj'}$ is 
equal to $\pm v_\bfj$ and in particular it does not vanish;
since $\bfj\cup\bfj'$ contains $j_k$ and $j_k>d$,
we deduce that $v\wedge w$ is not a $(k+h)$-vector in $W$, 
that is, $\Span(v\wedge w)\not\subset W$, as claimed.

\passo{Step~2: condition~\ref{lem:span0-2} implies condition~\ref{lem:span0-1}.}
More precisely, condition~\ref{lem:span0-1} holds with
$W':=W$ and $h:=d'-k+1=d-k+1$.

\passo{Step~3: condition~\ref{lem:span0-3} implies condition~\ref{lem:span0-1}.}
More precisely, condition~\ref{lem:span0-1} holds with
$W':=\Span(w)$: notice indeed that since $w$ is simple then 
$d':=\dim(W')=h$, and therefore the $h$-vectors in $\esterno^h(W')$
are just multiples of $w$.
\end{proof}

\begin{parag}[A Hodge-type operator]
\label{def:hodge}
We consider the operator $\hodge$ that acts on all vectors 
and covectors of $\R^n$, and more precisely maps $k$-vectors 
into $(n-k)$-covectors and vice versa, 
and is defined by the following property: for every 
$v \in \esterno_k(\R^n)$ and every $\alpha \in \esterno^k(\R^n)$ 
there holds
\[
\hodge v := v\antitrace \de x
\, , \quad
\hodge\alpha := e\trace \alpha
\, , 
\]
where $\de x:=\de x_1\wedge\dots\wedge\de x_n$
and $e:=e_1\wedge\dots\wedge e_n$.%
\footnoteb{This operator is similar to the standard Hodge star operator
but not exactly the same; it is defined in \cite{Federer}, \S1.5.2, 
but denoted there with a different symbol.}

Note that the definition of the interior products 
(see~\S\ref{def:int_prod}) yields
\[
\scalar{w}{\hodge v}= \scalar{w\wedge v}{\de x}
\, , \quad
\scalar{\hodge\alpha}{\beta}= \scalar{e}{\alpha\wedge\beta}
\, , 
\]
for every $(n-k)$-vector $w$ and every $(n-k)$-covector $\beta$.

Moreover for every $\bfi\in I(n,k)$ one has
\begin{equation}
\label{hodge_ei}
\hodge{e_\bfi} = \sign(\bfj,\bfi) \, \de x_{\bfj}
\, , \quad
\hodge{\de x_\bfi} = \sign(\bfi,\bfj) \, e_{\bfj}
\, , 
\end{equation}
where $\bfj$ is the multi-index in $I(n, n-k)$ 
consisting of all indices which are not in $\bfi$, 
and $\sign(\bfj, \bfi)$ is the sign of the permutation 
that reorders the sequence of indices 
$j_1, \dots, j_{n-k},i_1, \dots, i_k$.
The identities in \eqref{hodge_ei} show that
$\hodge$ is an involution, that is, 
$\hodge(\hodge v)=v$ and $\hodge(\hodge\alpha)=\alpha$.

Among the many identities relating $\hodge$ and the various
products, we will use the following one: for every 
$k$-vector $v$ and every $h$-covector $\alpha$ with $h\le k$ 
one has
\begin{equation}
\label{hodge_prod}
\hodge(v\trace\alpha)=(\hodge v)\wedge\alpha
\, .
\end{equation}
\end{parag}

\sottosezione*{Forms, vectorfields, currents}
Here we review the basic definitions and results concerning 
differential forms, vectorfields and currents.
These objects will be defined on a general open 
set $\Omega$ in $\R^n$, $n\ge 2$. 

Elementary introductions to 
the theory of currents can be found for instance in \cite{Kra-Par}, 
\cite{Simon}; the most complete reference remains~\cite{Federer}.

\begin{parag}[Forms and vectorfields]
A $k$-form is a map $\omega:\Omega\to\esterno^k(\R^n)$, 
and we sometime write it in terms of the canonical basis of
$\esterno^k(\R^n)$:
\[
\omega(x) = \sum_{\bfi\in I(n,k)} \omega_\bfi(x) \, \de x_\bfi
\, .
\]
Similarly, a $k$-vectorfield is a map 
$v:\Omega\to\esterno_k(\R^n)$, and we sometime write it as
\[
v(x) = \sum_{\bfi\in I(n,k)} v_\bfi(x) \, e_\bfi
\, .
\]
\end{parag}

\begin{parag}[Exterior derivative and divergence]
\label{def:diffop}
If $\omega$ is a $k$-form of class $C^1$, 
the \emph{exterior derivative} $\de\omega$ is the $(k+1)$-form 
defined in coordinates by the usual formula:
\begin{equation}
\label{e:diff}
\de\omega(x)
:=\sum_{\bfi\in I(n,k)} \sum_{j=1}^n
   \frac{\bd\omega_\bfi}{\bd x_j}(x) \, \de x_j\wedge\de x_\bfi
= \sum_{j=1}^n \de x_j \wedge \frac{\bd\omega}{\bd x_j}(x)
\, .
\end{equation}

If $v$ is a $k$-vectorfield of class $C^1$,
the \emph{divergence} $\dive v$ is the $(k-1)$-vectorfield
defined by
\begin{equation}
\label{e:div}
\dive v(x)
:=\sum_{\bfi\in I(n,k)}\sum_{j=1}^n 
      \frac{\bd v_\bfi}{\bd x_j}(x)\, e_\bfi \trace\de x_j
=\sum_{j=1}^n \frac{\bd v}{\bd x_j}(x)\trace\de x_j
\,.
\end{equation}

Definition \eqref{e:div} is not as standard as \eqref{e:diff}: 
we refer to \cite{Federer}, \S4.1.6, for the abstract characterization 
of the divergence operator and for the following identity, 
which relates divergence and exterior derivative (it can also be 
proved using \eqref{hodge_ei} and \eqref{hodge_prod}):
\begin{equation}
\label{e:div2}
\dive v:=(-1)^{n-k} \hodge(\de(\hodge v))
\,.
\end{equation}
For $k=1$, formula \eqref{e:div} reduces to the usual 
definition of divergence of a vectorfield 
(recall that $e_i\trace\de x_j=\scalar{e_i}{\de x_j}=\delta_{ij}$).
\end{parag}

\begin{parag}[Leibniz rules]
\label{leibniz}
The exterior derivative satisfies a Leibniz rule
\wrt the exterior product:
given a $k$-form $\omega$ and  a $k'$-form $\omega'$ on $\Omega$, 
both of class $C^1$, one has 
\begin{equation}
\label{e:leibniz}
\de\left(\omega\wedge\omega'\right)
=\de\omega\wedge\omega' + (-1)^k\, \omega\wedge\de\omega'
\, .
\end{equation}

The divergence satisfies a Leibniz rule 
\wrt the interior product: given a $k$-vector $v$ and  
an $h$-form $\omega$ on $\Omega$, both of class $C^1$ and with 
$h\le k$, one has
\begin{equation}
\label{e:leibniz-div}
\dive(v\trace\omega)
= (-1)^h \left((\dive v)\trace\omega + v\trace\de\omega\right)
\, ,
\end{equation}
which follows from \eqref{e:leibniz} using \eqref{hodge_prod}
and~\eqref{e:div2}.%
\footnoteb{Formulas relating divergence and exterior 
product (or exterior derivative and interior product) 
are more complicated, see~\S\ref{def:lie}.}
\end{parag}

\begin{parag}[Lie bracket and Cartan's formula]
\label{def:lie}
Given two vectorfields $v$, $v'$ on $\Omega$ of class $C^1$, 
the Lie bracket $[v,v']$ is the vectorfield on $\Omega$ 
defined by
\[
[v,v'](x)
:= \frac{\bd v}{\bd v'}(x) - \frac{\bd v'}{\bd v}(x)
= \de_x v \, (v'(x)) - \de_x v' \, (v(x))
\, , 
\]
where $\de_x v$ and $\de_x v'$ stand for the differentials of $v$ and $v'$
at the point $x$, viewed as linear maps from $\R^n$ into itself.

Consider now a \emph{simple} $k$-vectorfield $v=v_1\wedge\dots\wedge v_k$ 
with $k\ge 2$ where each $v_i$ is a vectorfield of class $C^1$ on $\Omega$. 
Then the divergence of $v$ can be computed using the following version 
of Cartan's formula:
\begin{equation}
\label{e:cartan1}
\begin{aligned}
    \dive v
& = \sum_{i=1}^k (-1)^{i-1} \dive v_i \, \Big( \bigwedge_{j\neq i}v_j \Big) \\
& \qquad + \sum_{1\le i<i'\le k} (-1)^{i+i'-1} [v_i, v_{i'}]
       \wedge \Big( \bigwedge_{j\neq i, i'} v_j \Big)
\,.
\end{aligned}
\end{equation}
In particular for $k=2$ we have
\begin{equation}
\label{e:cartan2}
\dive (v_1\wedge v_2)
=(\dive v_1) \, v_2 - (\dive v_2) \, v_1 + [v_1,v_2]
\,.
\end{equation}
Formula~\eqref{e:cartan1} can be found, written in a dual 
form, in \cite{War}, Proposition~2.25(f); we recover the form
above using identity~\eqref{e:div2}.
\end{parag}

\begin{parag}[Currents]
\label{def:curr}
A $k$-dimensional current (or simply $k$-current)
$T$ on the open set $\Omega$ in $\R^n$ is a continuous 
linear functional on the space of smooth $k$-forms 
with compact support in~$\Omega$. 
The \emph{boundary} of $T$ is the $(k-1)$-current 
$\bd T$ on $\Omega$ defined by 
$\scalar{\bd T}{\omega} := \scalar{T}{d\omega}$
for every smooth $(k-1)$-form $\omega$ with compact support.
The \emph{mass} of $T$, denoted by 
$\Mass(T)$, is the supremum of $\scalar{T}{\omega}$ over
all $k$-forms $\omega$ such that $|\omega(x)|\le 1$ for every~$x\in\Omega$.

By Riesz theorem, the fact that $T$ has finite mass 
is equivalent to saying that $T$ can be represented 
as a finite measure on $\Omega$ with values 
in the space $\esterno_k(\R^n)$, that is,
$T=\tau\mu$ where $\mu$ is a finite positive measure
on $\Omega$ and $\tau$ is a Borel $k$-vectorfield 
in~$L^1(\mu)$.
Thus
\[
\scalar{T}{\omega} 
= \int_{\Omega} \scalar{\tau(x)}{\omega(x)} \, d\mu(x)
\]
for every admissible $k$-form $\omega$ on $\Omega$,   
and $\Mass(T)=\int_{\Omega} |\tau| \, d\mu$.

\smallskip
Finally, a $k$-current $T$ is said to be:
\begin{enumeratethm}[label=\rm(\alph*)]
\item
\emph{normal} if both $T$ 
and $\bd T$ have finite mass;
\item
\label{def:curr-b}
\emph{rectifiable}
if $T=\tau m \Haus^k$ where $m$ is a
function in $L^1(\Haus^k)$ such that the set $E:=\{x: \, m(x)\ne 0\}$
is $k$-rectifiable, and $\tau$ is a simple $k$-vectorfield
with $|\tau|=1$ which spans the approximate
tangent space to $E$ at $x$ for $\Haus^k$-a.e.~$x\in E$;
\item
\emph{rectifiable with integer multiplicity} if the 
multiplicity $m$ is integer-valued;
\item
\emph{integral} if $T$ is rectifiable
with integer multiplicity and $\bd T$ 
has finite mass (if this is the case
then also $\bd T$  is rectifiable
with integer multiplicity).
\end{enumeratethm}
\end{parag}

\begin{remarks}\label{rem:curr}
\begin{enumeraterem}[ref={\ref*{rem:curr}(\roman*)}]
\item
\label{rem:curr-i}
When we write a current $T$ in the form 
$T=\tau\mu$ we always assume that $T$ has finite mass, 
$\mu$ is a (locally) finite positive measure, $\tau$ belongs to $L^1(\mu)$ 
and $\tau(x)\ne 0$ for $\mu$-a.e.~$x$; in particular
$\supp(T)=\supp(\mu)$.

\item
\label{rem:curr-ii}
The representation $T=\tau\mu$ 
is not unique. However, given another representation 
$T=\hat\tau\hat\mu$ we have that $\hat\mu=\rho\mu$ and $\hat\tau=\tau/\rho$
for some strictly positive function $\rho$. 
In particular $\mu$ and $\hat\mu$ are absolutely 
continuous \wrt each other.

\item
\label{rem:curr-iii}
The boundary operator and the (distributional) 
divergence operator are related by the formula 
$\bd T=-\dive T$.
More precisely, given a current of the form $T=\tau\,\Leb^n$, 
then $T$ is a normal current if and only if the 
divergence of $\tau$ belongs to $L^1(\Leb^n)$, 
and in that case $\bd T=-\dive\tau\, \Leb^n$.

\item
Given a $k$-current with finite mass $T=\tau\mu$ 
and a continuous $h$-form $\omega$ with $h\le k$, 
the \emph{interior product} of $T$ and $\omega$ 
is the $(k-h)$-current defined by
\begin{equation}
\label{e:intprod-curr}
T\trace\omega:=(\tau\trace\omega)\mu
\, .
\end{equation}
If $T$ is normal and $\omega$ is of class $C^1$, 
then the definition of boundary and \eqref{e:leibniz} 
give the following Leibniz rule:
\begin{equation}
\label{e:leibinz-curr}
\bd(T\trace\omega) 
= (-1)^h \big[ (\bd T) \trace\omega - T \trace d\omega \big]
\, .
\end{equation}
\end{enumeraterem} 
\end{remarks}

\sottosezione*{Distributions of $\boldsymbol k$-planes}
In this subsection we consider a distribution 
of $k$-planes $V$  defined on the open set $\Omega$, 
recall the notion of involutivity of $V$, 
define the distribution $\Vhat$ and the sets $N(V)$ and $N(V,d)$, 
and give some characterizations that are not widely used.

\begin{parag}[Distributions of $\boldsymbol k$-planes]
\label{def:distrib}
Let $1\le k\le n$.
A \emph{distribution of $k$-planes} on the open set $\Omega$ 
in $\R^n$ is a map $V$ that to every $x\in\Omega$ associates 
a $k$-dimensional plane $V(x)$ in $\R^n$, that is, 
a map from $\Omega$ to the Grassmannian $\Gr(k,n)$.

We say that a simple $k$-vectorfield $v=v_1\wedge\dots\wedge v_k$ 
\emph{spans} $V$ if for every $x\in \Omega$ one has
\[
V(x) 
= \Span(v(x))
= \Span\big\{ v_1(x),\dots,v_k(x) \big\}
\, .
\] 
Note that a distribution $V$  
of class $C^r$, with $r=0,1,\dots,\infty$,
is \emph{locally} spanned by $v=v_1\wedge\dots\wedge v_k$,
where the vectorfields $v_i$ are of class $C^r$.

We say that an $h$-vectorfield $w$ on $\Omega$
is \emph{tangent} to $V$ if $\Span(w(x)) \subset V(x)$ for every~$x$
(simply $w(x)\in V(x)$ when $h=1$).

We say that an $h$-current with finite mass $T=\tau\mu$ 
on $\Omega$ is \emph{tangent} to $V$ if $\Span(\tau(x)) \subset V(x)$
for $\mu$-a.e.~$x$.
Note that this notion does not depend on the choice
of $\tau$ and $\mu$ (recall Remark~\ref{rem:curr-ii}).
\end{parag}

\begin{remarks}\label{rem:tancurr}
\begin{enumeraterem}[ref={\ref*{rem:tancurr}(\roman*)}]
\item
\label{rem:tancurr-i}
If $T$ is a rectifiable $k$-current and $E$ is the associated 
rectifiable set (as in \S\ref{def:curr}\ref{def:curr-b}), 
the fact that $T$ is tangent to $V$ means
that $V(x)$ contains the approximate tangent space $T_xE$
for $\Haus^h$-a.e.~$x\in E$.
\item
\label{rem:tancurr-ii}
If $h=k$ and $V$ is spanned by a simple $k$-vectorfield $v$, then a current
$T$ with finite mass is tangent to $V$ if and only if it can be written 
as $T=v\mu$ for some \emph{signed} measure~$\mu$.
\end{enumeraterem}
\end{remarks}

\begin{parag}[Involutivity of $\boldsymbol{V}$
and the set $\boldsymbol{N(V)}$]
\label{def:inv}
Let $V$ be a distribution of $k$-planes 
of class $C^1$ on the open set $\Omega$ in $\R^n$.

We say that $V$ is \emph{involutive at $x\in\Omega$} if
for every couple of vectorfields $w$, $w'$ of class $C^1$ which
are tangent to $V$ 
the commutator $[w, w'](x)$ belongs to $V(x)$.
We say that $V$ is involutive if it is involutive 
at every point of~$\Omega$.

The set of all points $x$ where $V$ is not involutive is called
the \emph{non-involutivity set} of $V$ and denoted by $N(V)$.
Note that this set is open. 
\end{parag}

\begin{remark}
The involutivity of a distribution $V$ is most often 
defined in terms of the commutators of a given family of 
vectorfields $v_1,\dots,v_k$ that span $V$; the definition above is equivalent 
(see Corollary~\ref{cor:inv}) and has the slight advantage 
of being independent of the choice of the vectorfields $v_i$.
\end{remark}

\begin{parag}[The distribution $\boldsymbol\Vhat$
and the sets $\boldsymbol{N(V,d)}$]
\label{def:vhat}
Given $V$ as in \S\ref{def:inv}, 
for every $x\in\Omega$ we denote by $\Vhat(x)$ the subspace of
$\R^n$ spanned by all vectors in $V(x)$ and by the commutator 
(evaluated at $x$) of every couple of vectorfields $w$, $w'$ 
of class $C^1$ which are tangent to $V$, that is,
\[
\Vhat := 
V + \Span\big\{ [w, w'] \colon\text{$w, w'$ are tangent to $V$}\big\}  
\, .
\]
For every $d=k, \dots, n$ we set
\[
N(V,d) := \big\{ x\in\Omega \colon \dim(\Vhat(x))=d \big\}
\, .
\]
Thus $N(V,k)$ is the set of all points where $V$ is involutive
and
\[
N(V)= \bigcup_{d=k+1}^n N(V,d)
\, .
\]
\end{parag}

\begin{proposition}
\label{prop:vhat}
Let $V$ and $\Vhat$ be as in \S\ref{def:vhat}, 
and assume that $V$ be spanned by $v=v_1\wedge\dots\wedge v_k$ 
where each $v_i$ is a vectorfield of class $C^1$ on $\Omega$.
We consider the following distributions of planes on $\Omega$:%
\begin{enumeratethm}
\item
$V_1 := \Span\big\{ \dive(w\wedge w') 
\colon\text{$w, w'$ are $C^1$-vectorfields tangent to $V$} \big\}$;
\item
$V_2 := \Span\big\{ [v_i, v_j] \colon 1\le i, j\le k \big\}$;
\item
$V_3 := \Span\big\{ \dive(v_i \wedge v_j) \colon 1\le i, j\le k \big\}$;
\item
$V_4 := \big\{ \dive w 
\colon \text{$w$ is a $2$-vectorfield of class $C^1$ tangent to $V$} \big\}$.
\end{enumeratethm}
Then
\begin{equation}
\begin{aligned}
  \Vhat 
  = V+V_1
& = V+V_2 = V+V_3 \\
& = V+V_4 
  = V +\Span(\dive v)
  \, .
\end{aligned}
\label{e:vhat}
\end{equation}
\end{proposition}

\begin{corollary}
\label{cor:inv}
Let $V$ and $v=v_1\wedge\dots\wedge v_k$ be as in the previous statement.
Then the following assertions are equivalent at every given point
$x\in\Omega$:
\begin{enumeratethm}
\item
\label{cor:inv-i}
$V$ is involutive at $x$;
\item
\label{cor:inv-ii}
$[v_i, v_j] \in V$ or, equivalently, $\dive(v_i\wedge v_j)\in V$
for every $1\le i,j\le k$;
\item
\label{cor:inv-iii}
$\Span(\dive v) \subset V$;
\item
\label{cor:inv-iv}
$v \wedge ((\dive v) \trace\de x_\bfi) = 0$
for every $\bfi\in I(n,k-2)$.
\end{enumeratethm}
\end{corollary}

\begin{proof}[Proof of Proposition~\ref{prop:vhat}]
The proof of \eqref{e:vhat} is divided in several steps.

\passo{Step~1: $V+V_2=V+V_3$ and $V+V_1 = \Vhat$.}
These equalities follow from the inclusion
\[
\dive(w\wedge w') - [w, w'] \in V
\, , 
\]
which holds for every pair of $1$-vectorfields $w, w'$ tangent to $V$, 
and is a consequences of formula \eqref{e:cartan2}.

\passo{Step~2: $V_4 \subset V+V_3$.}
Every $2$-vectorfield $w$ of class $C^1$ tangent to $V$ can be written as
\[
w=\sum_{1\le i<j\le k} a_{ij} \, v_i\wedge v_j
\]
for suitable $C^1$-functions $a_{ij}$. 
By applying formula \eqref{e:leibniz-div} to the $2$-vectorfields $v_i \wedge v_j$ 
and the $0$-forms $a_{ij}$ we obtain
\[
\dive w = \sum_{1\le i<j\le k} 
a_{ij} \dive(v_i\wedge v_j) + (v_i\wedge v_j)\trace \de a_{ij}
\, ;
\]
this formula immediately implies the claim.

\passo{Step~3: proof the first four equalities in~\eqref{e:vhat}}:
\begin{align*}
\Vhat 
& = V+ V_1
  \quad\text{by Step~1} \\
& \subset V+V_4 
  \quad\text{because $V_1\subset V_4$} \\
& \subset V+V_3
  \quad\text{by Step~2} \allowdisplaybreaks[3] \\
& =V+V_2
  \quad\text{by Step~1} \\
& \subset\Vhat \hphantom{{}+V_1}
  \quad\text{because $V_2\subset \Vhat$.}
\end{align*}

The last equality in~\eqref{e:vhat} follows by the next two steps.

\passo{Step~4: $\Span(\dive v) \subset V_4$.}
Every vector in $\Span(\dive v)$ can be written 
as $(\dive v)\trace\alpha$ for some $(k-2)$-covector $\alpha$
(see~\S\ref{def:span}\ref{def:span-v}), and by applying 
formula \eqref{e:leibniz-div} to $v$ and to the constant 
form $\alpha$ we obtain
\[
(\dive v)\trace\alpha = (-1)^k \dive(v\trace\alpha)
\, ,
\]
and the right-hand side clearly belongs to $V_4$.

\passo{Step~5: $V_4\subset V+\Span(\dive v)$.}
Since $V$ is spanned by the \emph{simple} $k$-vectorfield $v$, 
every $2$-vectorfield $w$ tangent to $V$ can be written 
as $v\trace\omega$ for some $(k-2)$-form $\omega$; 
then formula \eqref{e:leibniz-div} 
implies that $\dive w=\dive(v\trace\omega)$ 
belongs to $V+\Span(\dive v)$.
\end{proof}

\begin{proof}[Proof of Corollary~\ref{cor:inv}]
The equivalence of \ref{cor:inv-i}, \ref{cor:inv-ii} 
and \ref{cor:inv-iii} follows immediately from 
Proposition~\ref{prop:vhat}.

Let us prove the implication 
\ref{cor:inv-iii}$\,\Rightarrow\,$\ref{cor:inv-iv}.
Assertion~\ref{cor:inv-iii} means that $\dive v$ is a $(k-1)$-vector in $V$. 
Thus $(\dive v)\trace \de x_\bfi$ is a $2$-vector in $V$ 
and $v\wedge ((\dive v)\trace \de x_\bfi)$ is a $(k+1)$-vector in $V$,
and it must vanish because $V$ has dimension~$k$. 

Finally, let us prove the implication 
\ref{cor:inv-iv}$\,\Rightarrow\,$\ref{cor:inv-iii}.
Every vector in $\Span(\dive v)$ can be written as 
$(\dive v)\trace\alpha$ for some $(k-2)$-covector $\alpha$
(see~\S\ref{def:span}\ref{def:span-v}).
Thus \ref{cor:inv-iv} implies that $v\wedge ((\dive v)\trace\alpha)=0$, 
which in turn implies that $(\dive v)\trace\alpha$ 
belongs to the span of $v$, which is $V$
(here we use that $v$ is simple and nontrivial).
\end{proof}

\begin{parag}[A weak notion of involutivity]
\label{def:weakinv}
Corollary~\ref{cor:inv}\ref{cor:inv-iv} shows that the 
involutivity of a distribution $V$ spanned by a 
$k$-vectorfield $v$ is characterized by the equation
\begin{equation}
\label{e:weakinv}
v \wedge ((\dive v) \trace\de x_\bfi) = 0
\quad\text{for every $\bfi\in I(n,k-2)$,}
\end{equation}
which for $k=2$ reduces to $v \wedge \dive v = 0$.

We point out that equation \eqref{e:weakinv} makes sense
even if $v$ is less regular than $C^1$.
More precisely, the right-hand side of \eqref{e:weakinv}
is a well-defined distribution if $v$ and $\dive v$ 
belong, locally, to function spaces which are 
in duality (and are closed under multiplication by functions
of class $C^\infty_c$) and therefore one can define involutivity
for such classes of vectorfields. 

For example, it suffices that $v$ be continuous 
and $\dive v$ be a locally finite measure, or that
$v$ belong to the Sobolev class $\smash{H^s_\loc}$ 
for some $s\ge 0$ and $\dive v\in\smash{H^{-s}_\loc}$.
In particular it suffices that $v\in\smash{H^{\scriptscriptstyle 1/2}_\loc}$
(in this case $\dive v\in\smash{H^{\scriptscriptstyle -1/2}_\loc}$ 
because the divergence is a first-order differential operator).
\end{parag}

\sottosezione*{Decomposability bundle and sparseness of a measure}
In this subsection we briefly recall the notion of decomposability bundle
of a measure $\mu$, and show that the dimension of this bundle 
gives a lower bound on the degree of sparseness of~$\mu$,
expressed in terms of absolute continuity \wrt 
integral geometric measures $\I^d_t$.

\begin{parag}[Decomposability bundle of a measure]
\label{def:decbun}
Here we briefly sketch the definition of the 
\emph{decomposability bundle} of a measure, introduced in 
\cite{Alb-Mar},~\S2.6.
Given a positive finite measure $\mu$ on the open set $\Omega$, 
we denote by $\F(\mu)$ the class of all families $\{F_t\colon t\in I\}$ 
parametrized by $I:=[0,1]$ such that:
\begin{itemize}
[leftmargin=25pt, itemsep=2pt]
\item
each $F_t$ is a $1$-dimensional rectifiable set in $\Omega$;
\item
the measure $\lambda:=\int_I (\Haus^1\trace F_t)\, dt$ 
satisfies $\lambda\ll\mu$.%
\footnoteb{Recall that $\lambda$ is given by 
$\lambda(E):=\int_I \Haus^1(F_t\cap E) \, dt$ for every Borel set 
$E\subset\Omega$, where $dt$ is the Lebesgue measure;
we implicitly require that this integral is well-defined and finite.}
\end{itemize}
The decomposability bundle of $\mu$ is a map that to every 
$x\in\Omega$ associates a (possibly trivial) linear subspace of $\R^n$, 
denoted in this paper by $W(\mu, x)$, which is uniquely determined 
up to $\mu$-null sets by the following properties:
\begin{enumeratethm}
\item
\label{def:decbun-i}
for every $\{F_t\}\in\F(\mu)$ there holds $T_xF_t\subset W(\mu, x)$ 
for $\Haus^1$-a.e.~$x\in F_t$ and a.e.~$t\in I$, 
where $T_xF_t$ is the approximate tangent line to the set $F_t$ at~$x$;
\item
\label{def:decbun-ii}
$W(\mu,\cdot)$ is $\mu$-minimal among all bundles $W(\cdot)$ that satisfy 
property \ref{def:decbun-i}, in the sense that $W(\mu, x) \subset W(x)$
for $\mu$-a.e.~$x$.
\end{enumeratethm}
\end{parag}

Besides some results already contained in \cite{Alb-Mar}, 
we will need the following statement, which is a consequence 
of a remarkable theorem by G.~De~Philippis and F.~Rindler~\cite{DP-Rin}.

\begin{proposition}
\label{prop:decomp}
Let $\mu$ and $W(\mu,\cdot)$ be as above, $d$ be an integer, 
and $E$ be a Borel set such that $\dim(W(\mu, x)) \ge d$ 
for $\mu$-a.e.~$x\in E$.
Then $\mu\trace E \ll \I^d_t \ll \Haus^d$.
\end{proposition}

For the proof we need the following two lemmas.

\begin{lemma}
\label{lem:decomp2}
Let $\mu$ and $W(\mu,\cdot)$ be as above, 
let $f:\Omega\to\R^m$ be a map of class $C^1$, and let $f_\#\mu$
be the pushforward of $\mu$ according to~$f$.
Then
\begin{equation}
\label{e:decomp2-1}
\de_x f(W(\mu, x)) \subset W(f_\#\mu, f(x))
\quad\text{for $\mu$-a.e.~$x\in\Omega$,}
\end{equation}
where $\de_xf:\R^n\to\R^m$ is the differential of $f$ at~$x$.
\end{lemma}

\proof[Sketch of proof]
We argue by contradiction and assume that there exists 
a Borel set $F$ with $\mu(F)>0$ where the inclusion 
in \eqref{e:decomp2-1} fails.
Then we can find a bounded Borel vectorfield $\tau$ on $\Omega$
such that 
\begin{enumeratethm}[label=\rm(\alph*)]
\item
\label{e:decomp2-2}
$\de_xf(\tau(x)) \notin W(f_\#\mu, f(x))$ 
and $|\tau(x)|=1$ for $\mu$-a.e.~$x\in F$;
\item
\label{e:decomp2-3}
$\tau(x)\in W(\mu, x)$ for $\mu$-a.e.
\end{enumeratethm}

\passo{Step~1.}
Using property~\ref{e:decomp2-3} and Corollary~6.5 in~\cite{Alb-Mar} 
we find a normal $1$-current $T$ 
such that $\tau$ is the Radon-Nikod{\'y}m density of $T$ \wrt~$\mu$, that is, 
$T=\tau\mu+T_s$ with $T_s$ singular \wrt~$\mu$. Possibly 
removing from $F$ a $\mu$-null subset where $T_s$ is concentrated, 
we can assume that $T\trace F=\tau\mu\trace F$.

\passo{Step~2.}
Using Theorem~5.5 in~\cite{Alb-Mar} we find a family of $1$-dimensional 
rectifiable sets $\{F_t \colon t\in I\}$ with $I:=[0,1]$
such that for every~$t\in I$
and $\Haus^1$-a.e.~$x\in F_t$ the approximate tangent space $T_xF_t$ 
is spanned by $\tau(x)$, 
and $\int_I (\Haus^1\trace F_t) \, dt = \mu\trace F$.

\passo{Step~3.}
For every $t \in I$ the set $E_t:=f(F_t)$ is rectifiable, 
the approximate tangent space $T_{f(x)} E_t$ is spanned 
by $\de_x f(\tau(x))$ for $\Haus^1$-a.e.~$x\in F_t$, 
the measures $\Haus^1\trace E_t$ and $f_\#(\Haus^1\trace F_t)$
are absolutely continuous \wrt each other, and
\[
\int_I ( \Haus^1\trace E_t )\, dt
\ll \int_I f_\#(\Haus^1\trace F_t) \, dt
= f_\#(\mu\trace B) \le f_\#\mu 
\, .
\]
Thus the family $\{E_t \colon t\in I\}$
belongs to $\F(f_\#\mu)$ and therefore 
property~\ref{def:decbun-i} in ~\S\ref{def:decbun}
implies that, for a.e.~$t\in I$ and 
$\Haus^1$-a.e.~$x\in F_t$,
\[
\de_x f(\tau(x)) \in T_{f(x)} E_t 
\subset W(f_\#\mu, f(x))
\, , 
\]
which means that $\de_x f(\tau(x))\in W(f_\#\mu, f(x))$
for $\mu$-a.e.~$x\in F$, in contradiction with
property~\ref{e:decomp2-2} above.
\qed

\begin{lemma}
\label{lem:decomp1}
Let $\mu$ and $W(\mu, \cdot)$ be as above, and assume that 
$\dim(W(\mu, x)) \ge d$ for $\mu$-a.e.~$x$.
Then $\mu\ll\I^d_t$.
\end{lemma}

\begin{proof}
We first introduce some notation:
\begin{itemize}
[leftmargin=25pt, itemsep=2pt]
\item
$\lambda_d$ is the Haar measure on the Grassmannian $\Gr(d,n)$;
\item
for every $V \in \Gr(d,n)$, $p_V:\R^n\to V$
is the orthogonal projection onto~$V$
and $\mu_V$ is the pushforward of
the measure $\mu$ according to~$p_V$.
\end{itemize}

Using the characterization of $\I^d_t$-null sets given in 
\S\ref{def:intgeo} it is easy to show that the assertion 
$\mu\ll\I^d_t$ 
is implied by the assertion $\mu_V\ll\Haus^d$ 
for $\lambda_d$-a.e.~$V$. The proof of the latter is 
divided in four steps.

\passo{Step~1: if $W\in \Gr(d',n)$ with $d'\ge d$ 
then $p_V(W)=V$ for $\lambda_d$-a.e.~$V$.}
Possibly replacing $W$ with a subspace, 
we can assume that $W$ has dimension~$d$.
Since $\ker(p_V)=V^\perp$, we have that
\[
p_V(W)=V 
\ \text{if and only if}\
\dim(W\cap V^\perp)=0
\, .
\]
Therefore, taking into account that the map $V\mapsto V^\perp$ 
is a bijection from $\Gr(d,n)$ to $\Gr(n-d,n)$ that preserves 
the respective Haar measures, we can reformulate the claim 
as follows: 
\[
\dim(W\cap Z)=0
\quad\text{for $\lambda_{n-d}$-a.e.~$Z\in\Gr(n-d,n)$.}
\]
This is equivalent to saying that the set 
\[
S_k:=\big\{ Z\in\Gr(n-d,n) \colon \dim(W\cap Z)=k \big\}
\]
is $\lambda_{n-d}$-null for every $k>0$, which is a consequence
of the fact that $S_k$ is actually a smooth submanifold of $\Gr(n-d,n)$ 
with dimension strictly lower than $\Gr(n-d,n)$.

\passo{Step~2: for $\lambda_d$-a.e.~$V$ one has
$p_V(W(\mu, x))=V$ for $\mu$-a.e.~$x\in\Omega$.}
By assumption we have $\dim(W(\mu, x))\ge d$ for 
$\mu$-a.e.~$x$, and then it suffices to use Step~1.

\passo{Step~3: for $\lambda_d$-a.e.~$V$ one has
\begin{equation}
\label{e:decomp1}
W(\mu_V, y)=V
\quad\text{for $\mu_V$-a.e.~$y\in V$.}
\end{equation}}
\indent
By applying Lemma~\ref{lem:decomp2} to the map $f:=p_V$
we obtain that, for every $d$-plane~$V$,
\[
W(\mu_V, p_V(x)) \supset p_V(W(\mu, x))
\quad\text{for $\mu$-a.e.~$x\in\Omega$,}
\]
and recalling Step~2 we obtain that, for $\lambda_d$-a.e.~$V$,
\[
W(\mu_V,p_V(x)) \supset p_V(W(\mu, x))=V
\quad\text{for $\mu$-a.e.~$x\in\Omega$,}
\]
which implies
\[
W(\mu_V, y) \supset V
\quad\text{for $\mu_V$-a.e.~$y\in V$}
\]
because $\mu_V$ is the pushforward of $\mu$ through $p_V$.
To obtain \eqref{e:decomp1} it is enough to recall that
$W(\mu_V, y) \subset V$ for $\mu_V$-a.e.~$y\in V$ because 
$\mu_V$ is a measure on~$V$.

\passo{Step~4: $\mu_V\ll\Haus^d$ for $\lambda_d$-a.e.~$V$.}
Identity~\eqref{e:decomp1} means the following:
if we identify the $d$-plane $V$ with $\R^d$ (isometrically), 
then $\mu_V$ is a measure on $\R^d$ 
whose decomposability bundle is a.e.~equal to $\R^d$, 
and therefore Corollary~1.12 and Lemma~3.1 in~\cite{DP-Rin}
imply that $\mu_V$ is absolutely continuous 
\wrt the Lebesgue measure on $\R^d$, that is, 
the restriction of $\Haus^d$ to~$V$.
\end{proof}

\begin{proof}[Proof of Proposition~\ref{prop:decomp}]
Let $\bar\mu$ be the restriction of $\mu$ to the set $E$.
By Proposition~2.9(i) in \cite{Alb-Mar} we have that 
$W(\bar\mu, x)=W(\mu, x)$ for $\bar\mu$-a.e.~$x$, 
which implies that $\dim(W(\bar\mu, x)) \ge d$
for $\bar\mu$-a.e.~$x$. 
We conclude the proof by applying Lemma~\ref{lem:decomp1}.
\end{proof}

\sottosezione*{Sparseness of a normal current and of its boundary}
In this subsection we establish a relation between
the degree of sparseness of a normal current $T$ and that of its boundary $\bd T$, 
both expressed in terms of absolute continuity \wrt Hausdorff measures.

\begin{proposition}
\label{prop:diffuse}
Let $T=\tau\mu$ be a normal $k$-current on the open set $\Omega$
in $\R^n$ with boundary $\bd T=\tau'\mu'$ such that
\begin{enumeratethm}[label={\rm(\alph*)}]
\item\label{prop:diffuse-i}
there exists a real number $\alpha\in [k,n]$ such that $\mu\ll\Haus^\alpha$.

\item\label{prop:diffuse-ii}
there exists a $C^1$-vectorfield $v$ on $\Omega$ 
such that $v\wedge\tau=0$~$\mu$-a.e.
\end{enumeratethm}
Let 
\[
E := \big\{ x\in\Omega \colon v(x) \wedge\tau'(x) \ne 0 \big\}
\, .
\]
Then $\mu'\trace E \ll\Haus^{\alpha-1}$.
\end{proposition}

\begin{remarks}\label{rem:diffuse}
\begin{enumeraterem}[ref={\ref*{rem:diffuse}(\roman*)}]
\item
\label{rem:diffuse-i}
If $\tau$ is simple, the condition 
$v\wedge\tau=0$ $\mu$-a.e.\ in assumption~\ref{prop:diffuse-ii}
is equivalent to $v(x)\in\Span(\tau(x))$ for $\mu$-a.e.~$x$, that is, 
$v$ is tangent to $T$.
\item
\label{rem:diffuse-ii}
If $k>1$ and $\tau'$ is simple, then 
$E=\{ x \colon v(x) \notin\Span(\tau'(x))\}$.
\item
\label{rem:diffuse-iii}
If $k=1$ then $\tau'$ is a real function with $\tau'\ne 0$ $\mu'$-a.e.,
and therefore $E$ can be equivalently defined as 
$E=\{ x \colon v(x) \ne 0\}$.
\end{enumeraterem}
\end{remarks}

Before the proof we present two examples that illustrate
the optimality of this statement:
the first one shows that the vectorfield $v$ cannot be just 
continuous, and the second one shows that the 
Hausdorff measures 
cannot be replaced by the integral geometric measures.

\begin{parag}[Example]
For every $a\in[0,1]$ let $T_a$ be the integral $1$-current in $\R^2$ 
associated to the (oriented) curve parametrized by 
$\gamma_a(t):=(t, at^2)$ with $t\in [0,1]$,
and let $T$ be the normal $1$-current given by the superposition
of all $T_a$, that is,
\[
\scalar{T}{\omega}:=\int_0^1 \scalar{T_a}{\omega} \, da
\quad\text{for every $1$-form $\omega$ of class $C^\infty_c$.}
\]
Then
\[
T= \Big(\frac{1}{x_1^2},\frac{2x_2}{x_1^3}\Big) \, \Leb^2\trace F 
\, , \quad
\bd T = \Haus^1 \trace I - \delta_0
\, ,
\]
where $\delta_0$ is the Dirac mass at $0$, and $F, I$ are the sets in $\R^2$ defined by
\begin{align*}
F & :=\big\{x\colon 0\le x_1\le 1\, , \ 0\le x_2\le x_1^2\big\} \, , \\ 
I & :=\big\{x\colon x_1=1\, , \ 0\le x_2\le 1\big\} \, .
\end{align*}
Thus $\mu\ll\Leb^2=\Haus^2$ but $\mu'\not\!\ll\Haus^1$.

Let now $v:F\to\R^2$ be the vectorfield given by 
$v(x):=(1,2x_2/x_1)$ if $x\ne 0$ and $v(0):=(1,0)$. 
One can check that $v$ is of class $C^{0,1/2}$ on $F$ 
and thus can be extended to the entire $\R^2$ with the same regularity. 
Moreover $v$ is tangent to $T$ and never vanishes on the set $F$, 
which contains the support of $\mu'$; therefore
the set $E$ contains the support of $\mu'$ 
(see Remark~\ref{rem:diffuse-iii})
and thus $\mu'\trace E=\mu'\not\!\ll\Haus^1$, 
which means that Proposition~\ref{prop:diffuse} 
fails for this choice of $v$.
(On the other hand, every vectorfield $v$ of class $C^1$ 
tangent to $T$ must vanish at $0$, thus for such $v$ the set $E$ 
does not contain $0$ and $\mu'\trace E \ll \Haus^1$, in accordance with 
Proposition~\ref{prop:diffuse}.)
\end{parag}

\begin{parag}[Example]
Take a Borel function $g:[0,1]\to\R$ whose graph $\Gamma$ 
is purely unrectifiable and $\Haus^1$-finite, and for every
$r\in\R$ consider the map $f_r:[0,1]\to\R^2$ given by $f_r(s):=(s, g(s)+r)$.
For every $a\in [0,1]$ we denote by $T_a$ the $1$-current in $\R^2$ associated 
to the (oriented) vertical segment $I_a:=[f_0(a),f_1(a)]$, 
and by $T$ the superposition of all such $T_a$ (defined as in the previous
example).
Then
\[
T=e \, \Leb^2\trace F
\, , \quad
\bd T = \lambda_1-\lambda_0
\, , 
\]
where $e:=(0,1)$, $F$ is the union of the segments $I_a$ with $0\le a\le 1$,
and $\lambda_r$ is the pushforward of the Lebesgue measure on $[0,1]$ 
according to the map~$f_r$.

Thus $\mu\ll\Leb^2=\Haus^2=\I^2_t$.
Moreover the constant vectorfield $v(x):=e$ is tangent to $T$
and never vanishes, and then Remark~\ref{rem:diffuse-iii} yields
$\mu'\trace E=\mu'$.

On the other hand $\mu'$ is supported on the set 
$\Gamma\cup(\Gamma+e)$, which is $\Haus^1$-finite and purely 
unrectifiable, and therefore $\I^1_t$-null, 
which implies that $\mu'\trace E=\mu'$ is singular \wrt~$\I^1_t$.
(However $\mu'\ll\Haus^1$, as predicted by 
Proposition~\ref{prop:diffuse}.)
\end{parag}

We now pass to the proof of Proposition~\ref{prop:diffuse}.

Using a localization argument we reduce to the case
where $\Omega=\R^n$ and $T$ and $v$ have compact support.
Moreover we can assume that $|\tau'|=1$~$\mu'$-a.e.

\smallskip
In the proof we use the flow associated to the vectorfield $v$, namely the map 
$\Phi:\R\times\R^n\to\R^n$ defined by
\[
\Phi(0,x)=x
\, , \quad
\frac{\bd \Phi}{\bd t}(t,x)= v( \Phi(t,x))
\quad\text{for every $t\in\R$, $x\in\R^n$,}
\]
and we write $\Phi_t(x)$ for $\Phi(t,x)$.
Since $v$ is of class $C^1$ and compactly supported,
the map $\Phi$ is well-defined and of class $C^1$, and each map 
$\Phi_t$ is bi-Lipschitz and coincides with the identity out of
a compact set which does not depend on $t$.  

\begin{lemma}
\label{lemma:magic}
Let $T$ and $v$ be as in Proposition~\ref{prop:diffuse}, 
and let $\Phi$ be as above. 
Given $t_0, t_1\in\R$, let $\ic{t_0,t_1}$ be the $1$-current in $\R$ associated
to the oriented interval $[t_0,t_1]$.
Then the pushforward of the product current $\ic{t_0,t_1}\times \bd T$
on $\R\times\R^n$ according to the map $\Phi$ satisfies
\begin{equation}
\label{e:magic}
\Phi_\# \big(\ic{t_0,t_1}\times \bd T \big) = (\Phi_{t_1})_\# T - (\Phi_{t_0})_\# T
\, .
\end{equation}
\end{lemma}

\begin{proof}
The homotopy formula (see for instance~\cite{Kra-Par}, \S7.4.3) states that
\[
\bd\big( \Phi_\#(\ic{t_0,t_1}\times T) \big) 
= (\Phi_{t_1})_\# T - (\Phi_{t_0})_\# T - \Phi_\# \big(\ic{t_0,t_1}\times \bd T \big)
\, , 
\] 
and then \eqref{e:magic} is implied by 
\begin{equation}
\label{e:magic0}
\Phi_\# \big(\ic{t_0,t_1}\times T \big) = 0
\, .
\end{equation}

The proof of \eqref{e:magic0} is divided in two steps.
We denote by $\de_{(t,x)}\Phi$ 
the differential of $\Phi$ at the point $(t,x)$, viewed 
as a linear map from $\R\times\R^n$ to $\R^n$, 
and let $e_0:=(1,0)\in\R\times\R^n$.
Moreover we tacitly identify $v\in\R^n$ with 
$(0,v)\in\R\times\R^n$, which yields an identification 
of $k$-vectors in $\R^n$ with $k$-vectors in $\R\times\R^n$.

\passo{Step~1: for every $t$ and $\mu$-a.e.~$x$ the pushforward
of the $(k+1)$-vector $e_0\wedge\tau(x)$ according to the linear map
$\de_{(t,x)}\Phi$ is null, that is
\begin{equation}
\label{e:magic0.1}
\big(\de_{(t,x)}\Phi\big)_\# (e_0\wedge\tau(x)) = 0
\, .
\end{equation}}%
Differentiating the semigroup identity
$\Phi(t+s,x) = \Phi(t,\Phi(s,x))$ \wrt~$s$
at $s=0$ we get
\[
\de_{(t,x)}\Phi(e_0) = \de_{(t,x)}\Phi(v(x))
\, ,
\]
and then 
\[
(\de_{(t,x)}\Phi)_\# (e_0\wedge \tau(x))
=(\de_{(t,x)}\Phi)_\# (v(x)\wedge \tau(x))
\, , 
\]
and using assumption~\ref{prop:diffuse-ii} in 
Proposition~\ref{prop:diffuse} we get~\eqref{e:magic0.1}.

\passo{Step~2: proof of \eqref{e:magic0}.}
Using \eqref{e:magic0.1}, for every test $(k+1)$-form $\omega$ on $\R\times\R^n$
we obtain
\begin{align*}
	\bigscalar{\Phi_\# \big( & \ic{t_0,t_1} \times T \big)}{\omega}
  = \bigscalar{\ic{t_0,t_1}\times T}{\Phi^\#(\omega)} \\
& = \int_{t_0}^{t_1} \int_{\R^n}
	\bigscalar{\omega(\Phi(t,x))}{\big(\de_{(t,x)}\Phi\big)_\# (e_0\wedge\tau(x))}
	\, d\mu(x) \, dt 
	= 0
	\, , 
\end{align*}
and \eqref{e:magic0} is proved.
\end{proof}

\begin{proof}[Proof of Proposition~\ref{prop:diffuse}]
We argue by contradiction and assume that there exists 
a compact set $E'\subset E$ such that
\[
\Haus^{\alpha-1}(E')=0
\, , \quad
\mu'(E')>0 
\, .
\]
Next we choose a point $x_0\in E'$ where the map 
$v\wedge\tau'$ is approximately continuous and the set $E'$ has density 
$1$ (in both cases the underlying measure is $\mu'$).
We fix for the time being $\delta, r>0$ and consider the sets
\[
E'':=E'\cap \overline{B(x_0,r)}
\, , \quad
F:=[0,\delta]\times E''
\, , \quad
G:=\Phi(F)
\, ,
\]
and the $k$-current
\[
S:=\Phi_\# \big( \ic{0,\delta}\times \bd T \big)
\, .
\]
We claim that
\begin{enumeratethm}
\item\label{s1}
$\Haus^\alpha(G)=0$; 
\item\label{s2}
$S$, viewed as a vector-vauled measure, satisfies $S\ll\Haus^\alpha$;%
\item\label{s3}
$S(G) \ne 0$ for $\delta$ and $r$ suitably chosen. 
\end{enumeratethm}
Note that \ref{s1} and \ref{s2} imply $S(G)=0$, which contradicts \ref{s3}.
To conclude the proof it remains to prove claims \ref{s1}--\ref{s3}.

\passo{Step~1: proof of~\ref{s1}.}
We have the following chain of implications:
$\Haus^{\alpha-1}(E')=0$
$\Rightarrow$
$\Haus^{\alpha-1}(E'')=0$
$\Rightarrow$
$\Haus^{\alpha}(F)=0$
$\Rightarrow$
$\Haus^{\alpha}(G)=0$
(the second implication follows from \cite{Federer}, Theorem~2.10.45; the third one 
from the fact that $\Phi$ is Lipschitz).

\passo{Step~2: proof of~\ref{s2}.}
Using the definition of the current $S$ and Lemma~\ref{lemma:magic} we obtain
$S=(\Phi_\delta)_\# T - T$; to conclude we recall that $T\ll\mu\ll\Haus^\alpha$ 
by assumption, and observe that 
 $(\Phi_\delta)_\# T\ll(\Phi_\delta)_\# \mu \ll \Haus^\alpha$
because the map $\Phi_\delta$ is bi-Lipschitz.

\passo{Step~3: proof of~\ref{s3}.}
To prove this claim we set $\rho(r):= \mu'\big( \overline{B(x_0,r)} \big)$ 
and show that 
\begin{equation}
\label{e:diffuse0}
\lim_{r\to 0}\lim_{\delta\to 0} \frac{S(G)}{\delta\rho(r)} 
= v(x_0)\wedge \tau'(x_0) \ne 0
\, .
\end{equation}
Let $\lambda$ be the product measure on $\R\times\R^n$ given by 
$\lambda:=(\Leb^1\trace [0,\delta])\times \mu'$, and let
$g:\R\times\R^n\to\esterno_k(\R^n)$ be the map given 
by $g(t,x):= (\de_{(t,x)}\Phi)_\# (e_0\wedge \tau'(x))$.
Using the definition of pushforward of currents we obtain
\begin{equation}
\label{e:diffuse1}
S(G) = \int_{\Phi^{-1}(G)} g\, d\lambda
= A_1 + A_2 + A_3
\end{equation}
where
\begin{align*}
A_1 & := \int_{F} g(0,x) \, d\lambda(t,x) = \delta \int_{E''} g(0,x) \, d\mu'(x) \, , \allowdisplaybreaks[2]  \\
A_2 & := \int_{F} \big( g(t,x)-g(0,x) \big)\, d\lambda(t,x) \, , \\
A_3 & := \int_{\Phi^{-1}(G)\setminus F} g(t,x) \, d\lambda(t,x) \, .
\end{align*}
The definitions of $g$ and $\Phi$ yield
\[
g(0,x)=v(x)\wedge \tau'(x)
\, .
\]
Using this identity, the definitions of $\rho(r)$ and $E''$, and the choice of $x_0$ 
we obtain that 
\begin{equation}
\label{e:diffuse3}
\frac{A_1}{\delta\rho(r)} 
= \frac{1}{\rho(r)} \int_{\overline{B(x_0,r)}\cap E'}  \hskip - 5pt v\wedge \tau' \, d\mu'
\, \mathop{\longrightarrow}\limits_{r\to 0} \, 
v(x_0) \wedge \tau'(x_0)
\, .
\end{equation}
Using the dominated convergence theorem and the fact that 
$g$ is continuous in $t$ and uniformly bounded we obtain
that, for every fixed $r$,
\begin{equation}
\label{e:diffuse5}
\frac{|A_2|}{\delta} \le \int_{\overline{B(x_0,r)}} \Big( \sup_{0\le t\le\delta} |g(t,x)-g(0,x)| \Big) \, d\mu'(x)
\, \mathop{\longrightarrow}\limits_{\delta\to 0} \, 
0
\, .
\end{equation}
Finally we let $E_\delta$ be the projection of $\Phi^{-1}(G) \cap \big([0,\delta]\times\R^n\big)$
onto $\R^n$ and notice that this is a closed set that contains $E''$ and for every fixed $r$
converges to $E''$ in the Hausdorff distance as $\delta\to 0$. 
Since $|g|$ is bounded by some constant $m$ (because so is~$\tau'$) 
we obtain that 
\begin{equation}
\label{e:diffuse6}
\frac{|A_3|}{\delta} 
\le \frac{m}{\delta} \,\lambda\big(\Phi^{-1}(G)\setminus F\big) 
\le m \, \mu'\big(E_\delta \setminus E''\big)
\, \mathop{\longrightarrow}\limits_{\delta\to 0} \, 
0
\, .
\end{equation}
Putting together \eqref{e:diffuse1}, \eqref{e:diffuse3}, \eqref{e:diffuse5} 
and \eqref{e:diffuse6} we obtain \eqref{e:diffuse0}.
\end{proof}

\section{The key identity}
\label{sec:main}
The main result in this section is identity~\eqref{e:d}
in Proposition~\ref{prop:form_d}.
Using this identity we obtain the fundamental
relation between the boundary of a normal $k$-current 
tangent to a distribution of $k$-planes $V$ and the 
distribution $\Vhat$ associated to $V$ 
(Theorem~\ref{thm:gpb-adv}).

\medskip
Through this section, $k$ and $n$ are integers 
that satisfy $2\le k<n$, $\Omega$ is an open set in $\R^n$, 
$V$ is a distribution of $k$-planes $V$ on $\Omega$ spanned 
by vectorfields $v_1,\dots, v_k$ of class $C^1$ on $\Omega$ 
and $v:=v_1\wedge\dots\wedge v_k$, as usual.

Moreover $T$ is a normal $k$-current on $\Omega$ which is tangent to $V$, 
which we write as $T=v\mu$ where $\mu$ is a suitable \emph{signed} measure
(see Remark~\ref{rem:tancurr-ii}).
In the sequel it is important to remember that $\mu$ is not necessarily
positive. 

As usual write $\bd T=\tau'\mu'$ where $\mu'$ is a positive
measure and $\tau'$ is a density with values in $(k-1)$-vectors.

\medskip
Notice that we assume that the distribution $V$ 
is \emph{globally} spanned by $k$ vectorfields, 
and not just locally (cf.~\S\ref{def:distrib}). 
There is however no loss of generality, 
because all statements we are interested in are
local in nature.

\begin{lemma}
\label{lem:alpha}
Take $v$ and $V$ as above and
consider the $(k-1)$-form 
\begin{equation}
\label{set_alpha}
\alpha:=\hodge(w\wedge u)
\end{equation}
where 
$u=u_1\wedge \dots\wedge u_{n-k-1}$ is a simple $(n-k-1)$-vector 
and $w=w_1\wedge w_2$ is a simple $2$-vectorfield on $\Omega$ 
with $w_1,w_2$ vectorfields of class $C^1$ tangent to~$V$. 
Then
\begin{enumeratethm}
\item
\label{lem:alpha-i}
$v\trace\alpha=0$ on $\Omega$;
\item
\label{lem:alpha-ii}
$\scalar{v}{\de\alpha}
=\scalar{v\wedge\dive w\wedge u}{\de x}$
on $\Omega$;
\item
\label{lem:alpha-iii}
$\scalar{v}{\de\alpha}\neq 0$ at every point of $\Omega$ 
where $v\wedge\dive w\wedge u\neq 0$.
\end{enumeratethm}
\end{lemma}

\begin{proof}
To prove \ref{lem:alpha-i} we show that
$\scalar{v\trace\alpha}{\beta}=0$ for every 
$1$-covector~$\beta$. Indeed 
\begin{align*}
  \scalar{v\trace\alpha}{\beta} 
  =\scalar{v}{\alpha\wedge\beta}
& =(-1)^{k-1} \scalar{v}{\beta\wedge\alpha} \\
& =(-1)^{k-1} \scalar{v\trace\beta}{\alpha} \\
& =(-1)^{k-1} \scalar{v\trace\beta}{\hodge(w\wedge u)} \\
& =(-1)^{k-1} \scalar{(v\trace\beta)\wedge w\wedge u}{\de x}
  =0 \, ,
\end{align*}
where the last equality holds because 
$(v\trace\beta)\wedge w$ is a $(k+1)$-vectorfield 
tangent to the distribution of $k$-planes $V$, 
and therefore it is everywhere null.

\smallskip
Let us prove~\ref{lem:alpha-ii}.
Using \eqref{e:div2} we get
\begin{align}
   \scalar{v}{\de\alpha}
  =\scalar{v}{\de(\hodge(w\wedge u))}
& =\scalar{v}{\hodge(\dive (w\wedge u))} 
  \nonumber \\
& =\scalar{v \wedge (\dive (w\wedge u))}{\de x}
  \, . \label{e:3.2}
\end{align}
Since both $w$ and $u$ are simple 
we can use formula \eqref{e:cartan1} to compute the 
divergence of $w\wedge u$, 
obtaining 
\begin{equation}
\label{e:3.3}
\dive (w \wedge u)
=[w_1,w_2]\wedge u + w'
\, , 
\end{equation}
where 
\begin{align*}
w'
& = (\dive w_1) \, w_2\wedge u - (\dive w_2) \, w_1\wedge u \\
& \qquad + \sum_{i=1}^{n-k-1} (-1)^i \, \big( [w_1, u_i] \wedge w_2 - [w_2, u_i] \wedge w_1\big)
    \wedge \Big( \bigwedge_{j\neq i} u_j \Big)
	\, .
\end{align*}
Now, each $w_i$ belongs to $V=\Span(v)$ by assumption,  
hence $v\wedge w_i=0$, which implies that $v\wedge w'=0$.
Therefore using \eqref{e:3.3} and \eqref{e:cartan2}
we get
\[
 v \wedge \dive (w\wedge u)
=v \wedge [w_1,w_2]\wedge u
=v \wedge \dive w \wedge u
\, . 
\]
Plugging this formula into \eqref{e:3.2} proves~\ref{lem:alpha-ii}.

Finally, \ref{lem:alpha-iii} 
is an immediate consequence of~\ref{lem:alpha-ii}.
\end{proof}

\begin{proposition}
\label{prop:form_d}
Take $v$, $V$, $T=v\mu$ and $\bd T=\tau'\mu'$ as above,
and let $w$ be a $2$-vectorfield of class $C^1$ on $\Omega$ which is tangent to~$V$.
Then the following identity of measures (with values in $(k+1)$-vectors)
holds:
\begin{equation}
\label{e:d}
(\tau'\wedge w) \, \mu' = (v\wedge\dive w) \, \mu
\, .
\end{equation}
\end{proposition}

\begin{proof}
The proof is divided in two cases.

\passo{Case~1: $w=w_1\wedge w_2$ with $w_1,w_2$ 
vectorfields of class $C^1$ tangent to $V$.}
Fix a simple $(n-k-1)$-vector $u$ 
and let $\alpha$ be the $(k-1)$-form defined in \eqref{set_alpha}. 
Recalling definition \eqref{e:intprod-curr}
and Lemma~\ref{lem:alpha}\ref{lem:alpha-i} we obtain that 
$T\trace\alpha=(v\trace\alpha)\mu=0$. 
Therefore formula \eqref{e:leibinz-curr} yields
\begin{align*}
0 = (-1)^{k-1} \bd(T\trace\alpha) 
& = \bd T\trace\alpha - T\trace\de\alpha \\
& = (\tau'\trace\alpha) \, \mu' - (v\trace\de\alpha) \, \mu \\
& = \scalar{\tau'\wedge w\wedge u}{\de x} \, \mu'
    - \scalar{v\wedge\dive w\wedge u}{\de x} \, \mu
\end{align*}
(in the last equality we used Lemma~\ref{lem:alpha}\ref{lem:alpha-ii}).
Hence
\[
(\tau'\wedge w \wedge u) \, \mu'
= (v\wedge\dive w\wedge u) \, \mu
\,,
\]
which implies \eqref{e:d} by the arbitrariness of $u$.

\passo{Case~2: $w$ is arbitrary.}
Then $w$ can be written in the form 
\[
w=\sum_{1\le i<j\le k} a_{ij} \, v_i\wedge v_j
\]
for suitable functions $a_{ij}$ of class $C^1$. 
Then identity~\eqref{e:d} holds for each
addendum $w_{ij}:=a_{ij} \, v_i\wedge v_j$ as just proved, 
and therefore it holds for $w$, too.
\end{proof}

Using formula \eqref{e:d} we can easily establish 
the following key relation between
the boundary of $T$ and the distribution
$\Vhat$ defined in~\S\ref{def:vhat}.

\begin{theorem}
\label{thm:gpb-adv}
Take $v$, $V$, $T=v\mu$ and $\bd T=\tau'\mu'$ as above,
and let $\mu'=\mu'_a+\mu'_s$ be the Lebesgue decomposition 
of the measure $\mu'$ \wrt $\mu$.
Then 
\begin{enumeratethm}
\item
\label{thm:gpb-adv-i}
$\Span(\tau'(x)) \subset V(x)$ for $\mu'_s$-a.e.~$x$;
\item
\label{thm:gpb-adv-ii}
$\Span(\tau'(x)) + V(x) = \Vhat(x)$ for $\mu'_a$-a.e.~$x$.
\end{enumeratethm}
\end{theorem}

\begin{proof}
We denote by $X(V)$ the space of
all $2$-vectorfield of class $C^1$ on $\Omega$ 
tangent to $V$ and write $\mu'_a=\rho\mu$ 
for a suitable density $\rho$.

Let $w \in X(V)$. Then equation \eqref{e:d} can be rewritten as
\begin{align}
\tau'\wedge w 
& =0 
\hskip 44.4 pt \quad\text{for $\mu'_s$-a.e.~$x$,} 
	\label{e:3.5} \\
\tau'\wedge w 
& = {\textstyle\frac{1}{\rho}} v\wedge\dive w
\quad\text{for $\mu'_a$-a.e.~$x$.} 
	\label{e:3.5.1}
\end{align}

The proof is now divided in three steps. The
first one contains the proof of statement~\ref{thm:gpb-adv-i}, 
while the others give statement~\ref{thm:gpb-adv-ii}.

\passo{Step~1: proof of statement~\ref{thm:gpb-adv-i}.}
Equation~\eqref{e:3.5} implies that
for every $w\in X(V)$ there exists a $\mu'_s$-null set $N_w$
such that $\tau'(x)\wedge w(x)=0$ for every $x\notin N_w$.
Take now a countable dense family $X'\subset X(V)$, 
and let $N$ be the union of $N_w$ over all $w\in X'$.
Then $N$ is $\mu'_s$-null and it is easy to check that
for every $x\notin N$ and every $w\in X(V)$ there holds
$\tau'(x)\wedge w(x)=0$, which means that
\[
\tau'(x)\wedge w=0
\quad\text{for every $2$-vector $w$ in $V(x)$,}
\]
and therefore 
$\Span(\tau'(x))\subset V(x)$ 
(apply Lemma~\ref{lem:span0} with assumption~\ref{lem:span0-2}).

\passo{Step~2: $\Span(\tau'(x))\subset\Vhat(x)$ for $\mu'_a$-a.e.~$x$.}
Let $w\in X(V)$. 
Using equation~\eqref{e:3.5.1} 
and the inclusion $\Span(v\wedge\dive w)\subset\Vhat$
(Proposition~\ref{prop:vhat})
we obtain
\[
\Span\big(\tau'(x)\wedge w(x) \big)\subset \Vhat(x)
\ \text{for $\mu'_a$-a.e.~$x$,}
\]
and proceeding as in Step~1 we find a $\mu'_a$-null set $N$ 
such that, for every $x\notin N$, 
\[
\Span\big(\tau'(x)\wedge w \big)\subset \Vhat(x)
\quad\text{for every $2$-vector $w$ in $V(x)$,}
\]
and then $\Span(\tau'(x))\subset \Vhat(x)$ 
(apply Lemma~\ref{lem:span0} with assumption~\ref{lem:span0-1}).

\passo{Step~3: $\Vhat(x)\subset V(x)+\Span(\tau'(x))$ for $\mu'_a$-a.e.~$x$.}
Let $w\in X(V)$. 
Using equation~\eqref{e:3.5.1}
and the inclusion $\Span(\tau'\wedge w) \subset \Span(\tau')+V$, 
we obtain that 
\[
\Span\big(v(x)\wedge \dive w(x) \big)
\subset \Span(\tau'(x))+V(x)
\ \text{for $\mu'_a$-a.e.~$x$.}
\]
Proceeding as in Step~1 we find a $\mu'_a$-null set $N$ 
such that, for every $x\notin N$ and every $w\in X(V)$,
\[
\Span\big(v(x)\wedge \dive w(x) \big)
\subset V(x) + \Span(\tau'(x))
\, , 
\]
and using Lemma~\ref{lem:span0} with assumption~\ref{lem:span0-3} 
we obtain that, for every $x\notin N$,
\[
\dive w(x) \in V(x)+\Span(\tau'(x))
\, .
\]
Then the claim follows using Proposition~\ref{prop:vhat}.
\end{proof}

\begin{remark}
\label{rem:gpb-adv}
Recall that $\Vhat(x)$ agrees with $V(x)$ for every $x$ in the 
involutivity set $N(V,k)=\Omega\setminus N(V)$, and strictly contains
$V(x)$ for every $x$ in the non-involutivity set $N(V)$.  
Then Theorem~\ref{thm:gpb-adv} implies that the inclusion that defines 
the geometric property of the boundary for $T$, namely
\[
\Span(\tau'(x)) \subset V(x)
\, ,
\]
holds for $\mu'_s$-a.e.~$x\in\Omega$
and for $\mu'_a$-a.e.~$x\in\Omega\setminus N(V)$, 
and does not hold for $\mu'_a$-a.e.\ $x\in N(V)$.
\end{remark}

\begin{remark}
The case $k=2$ and $n=3$ of Proposition~\ref{prop:form_d} is especially significant. 
In this case a form $\alpha$ with properties \ref{lem:alpha-i}-\ref{lem:alpha-iii}
in Lemma~\ref{lem:alpha} is simply given by $\alpha:=\hodge v$, and
equation \eqref{e:d} in Proposition~\ref{prop:form_d} reduces to
\[
(\tau'\wedge v) \, \mu' = (v\wedge\dive v) \, \mu
\, .
\]
\end{remark}

\section{Proofs of the results in Section~\ref{sec:intro}}
\label{sec:proofs}
In this section we prove 
Theorem~\ref{thm:diffuse}, 
Corollary~\ref{cor:normalfrob}, and Theorems~\ref{thm:frobenius}, 
\ref{thm:gpb-frob} and \ref{thm:ntv} (in this order).

We follow the notation of Section~\ref{sec:main}.
In particular $V$ is spanned by $v_1,\dots,v_k$, 
$v:=v_1\wedge,\dots\wedge v_k$, and $T=\tau\mu$ 
where $\mu$ is a suitable \emph{signed} measure.
Note that it is sufficient to prove the statements 
above for this measure $\mu$, despite the fact that
it may be not positive (cf.~Remark~\ref{rem:curr-ii}).

\begin{proof}[Proof of Theorem~\ref{thm:diffuse}]
We denote by $X(V)$ the space of
all $2$-vectorfield of class $C^1$ on $\Omega$ 
tangent to $V$.
The proof is divided in several steps.

\passo{Step~1: proof of statement~\ref{thm:diffuse-i}.}
Let $w\in X(V)$ and let 
$\mu=\mu_a+\mu_s$ be the Lebesgue decomposition 
of $\mu$ \wrt $\mu'$.
Then identity \eqref{e:d} yields $(v\wedge \dive w)\,\mu_s=0$, 
and therefore we can find a $|\mu_s|$-null set $N_w$ such that, 
for every $x\notin N_w$,
\[
v(x) \wedge \dive w(x) =0
\, .
\]
Proceeding as in Step~1 of the proof of Theorem~\ref{thm:gpb-adv}
we find a $|\mu_s|$-null set $N$ such that the previous 
equation holds for every $x\notin N$ and every $w\in X(V)$, 
and applying Lemma~\ref{lem:span0} with assumption~\ref{lem:span0-3}
we obtain 
\[
\dive w(x) \in V(x)
\, .
\]
By Proposition~\ref{prop:vhat} this means $V(x)=\Vhat(x)$
at every $x\notin N$, 
that is, $V$ is involutive at every~$x\notin N$
or, in other words, the non-involutivity set $N(V)$ is $|\mu_s|$-null, 
which finally implies $\mu\trace N(V) \ll\mu_a\ll\mu'$.

\medskip
For the next step we denote by $\bar\mu$ the 
restriction of $|\mu|$ to $N(V)$. 
Recall that $W(\bar\mu,\cdot)$ is the decomposability 
bundle of $\bar\mu$ (see~\S\ref{def:decbun}).

\passo{Step~2: $W(\bar\mu, x) \supset \Vhat(x)$ 
for $\bar\mu$-a.e.~$x$.}
Indeed, since $T=v\mu$ and $\bd T=\tau'\mu'$ are normal currents, 
Theorem 5.10 in \cite{Alb-Mar} implies that 
the decomposability bundles of the measures 
$|\mu|$ and $\mu'$ contain the span of $\tau$ and $\tau'$ 
respectively, that is,
\begin{align}
W(|\mu|,x) 
& \supset \Span(v(x)) = V(x)
  \quad\text{for $|\mu|$-a.e.~$x$,}
  \label{e:3.7} \\
W(\mu',x) 
& \supset \Span(\tau'(x)) 
  \hskip 46 pt \text{for $\mu'$-a.e.~$x$.}
  \label{e:3.8}
\end{align}
On the other hand $\bar\mu$ is absolutely continuous 
\wrt $|\mu|$ and also \wrt $\mu'$ (by statement~\ref{thm:diffuse-i}) 
and therefore Proposition~2.9(i) in \cite{Alb-Mar} yields
\begin{equation}
\label{e:3.9}
W(\bar\mu, x) = W(|\mu|, x) = W(\mu', x)
\quad\text{for $\bar\mu$-a.e.~$x$.}
\end{equation}
Putting together \eqref{e:3.7}, \eqref{e:3.8} and \eqref{e:3.9} 
we obtain
\[
W(\bar\mu, x) \supset V(x) + \Span(\tau'(x))
\quad\text{for $\bar\mu$-a.e.~$x$,}
\]
and we conclude the proof of the claim recalling that
$V+\Span(\tau')=\Vhat$ by Proposition~\ref{prop:vhat}.

\passo{Step~3: proof of statement~\ref{thm:diffuse-ii}.}
For every $d=k+1, \dots, n$ consider the measure
$\mu_d:=|\mu|\trace N(V,d)$. 
Since $\mu_d$ is absolutely continuous
\wrt $\bar\mu$, using Proposition~2.9(i) in \cite{Alb-Mar}
and Step~2 we obtain that
\[
W(\mu_d, x) = W(\bar\mu, x) \supset \Vhat(x)
\quad\text{for $\mu_d$-a.e.~$x$,}
\]
and in particular $\dim(W(\mu_d, x)) \ge d$ for $\mu_d$-a.e.~$x$.
We now conclude using Proposition~\ref{prop:decomp} 
or Lemma~\ref{lem:decomp1}.

\medskip
For the rest of the proof we fix $d=k+1, \dots, n$
and consider the following sets:
\begin{align*}
\Omega_d 
& :=\big\{x\in\Omega\colon \dim(\Vhat(x))\ge d \big\} 
	= N(V,d) \cup\dots\cup N(V,n) \, , \\
F
& :=\big\{x\in\Omega_d\colon V(x) \subset \Span(\tau'(x)) \big\} \, , \\
E_i
& :=\big\{x\in\Omega_d\colon v_i(x) \notin \Span(\tau'(x)) \big\}
\quad\text{with $i=1, \dots, k$.}
\end{align*}

\passo{Step~4: $\mu'\trace F \ll \Haus^d \ll \Haus^{d-1}$.}
Using the identity $\Vhat= V + \Span(\tau')$ (Proposition~\ref{prop:vhat}),
we have that for $x\in F$ the linear space $\Span(\tau'(x))$ contains $\Vhat(x)$ 
and therefore has dimension at least $d$.
Thus \eqref{e:3.8} implies that $W(\mu',x)$ has dimension at least $d$ 
for $\mu'$-a.e.~$x\in F$, and the claim follows from 
Proposition~\ref{prop:decomp}.

\passo{Step~5: $\mu'\trace E_i \ll \Haus^{d-1}$ for $i=1, \dots, k$.}
We prove this claim by applying Proposition~\ref{prop:diffuse} 
to the open set $\Omega_d$, the current $T=v\mu=\tau|\mu|$
and the vectorfield~$v_i$.
To this end we notice that 
\begin{itemize}
[leftmargin=25pt, itemsep=2pt]
\item
$|\mu|\trace\Omega_d\ll\Haus^d$ by statement~\ref{thm:diffuse-ii}
and the definition of~$\Omega_d$;
\item
$v_i\wedge v=0$ everywhere in $\Omega$ and then also in $\Omega_d$;
\item
$v_i\wedge \tau'\ne 0$ on $E_i$ by the definition of $E_i$.
\end{itemize}

\passo{Step~6: $\mu'\trace\Omega_d \ll \Haus^{d-1}$, which implies 
statement~\ref{thm:diffuse-iii}.}
To prove the first part of the claim 
we use that $\Omega_d=F\cup E_1\cup\dots\cup E_k$ and Steps~4 and~5.
To prove statement~\ref{thm:diffuse-iii} use that $N(V,d) \subset \Omega_d$.
\end{proof}

\begin{proof}[Proof of Corollary~\ref{cor:normalfrob}]
Since the set $N(V)$ is open, proving that the support of $\mu$ 
does not intersect $N(V)$ is equivalent to showing that 
$\mu\trace N(V)=0$.

If condition~\ref{cor:normalfrob-a} holds, 
then $\mu$ is singular \wrt $\I^d_t$
for every $d\ge k+1$, and this fact and 
Theorem~\ref{thm:diffuse}\ref{thm:diffuse-ii} imply $\mu\trace N(V,d)=0$, 
and since $N(V)$ is the union of all $N(V,d)$ with $d\ge k+1$,  
we obtain $\mu\trace N(V)=0$, as desired.

If condition~\ref{cor:normalfrob-b} holds then
Theorem~\ref{thm:diffuse}\ref{thm:diffuse-i} yields $\mu\trace N(V)=0$. 

To conclude the proof we notice that
condition~\ref{cor:normalfrob-c} implies
condition~\ref{cor:normalfrob-a}, 
and condition~\ref{cor:normalfrob-d} implies
condition~\ref{cor:normalfrob-b}.
\end{proof}

\begin{proof}[Proof of Theorem~\ref{thm:frobenius}]
Apply Corollary~\ref{cor:normalfrob} with condition~\ref{cor:normalfrob-c}.
\end{proof}

\begin{proof}[Proof of Theorem~\ref{thm:gpb-frob}]
We begin with the implication 
\ref{thm:gpb-frob-i}$\,\Rightarrow\,$\ref{thm:gpb-frob-ii}.
The geometric property of the boundary for $T$, namely inclusion \eqref{e:gpb}, 
and Theorem~\ref{thm:gpb-adv}\ref{thm:gpb-adv-ii} imply 
that $V=\Vhat$ $\mu'_a$-a.e., where $\mu'_a$ is the absolutely continuous
part of $\mu'$ \wrt~$\mu$. 

Since $V(x)\ne \Vhat(x)$ at every $x\in N(V)$ we infer that 
$\mu'_a\trace N(V)=0$ and using Theorem~\ref{thm:diffuse}\ref{thm:diffuse-i}
we deduce that $\mu\trace N(V)=0$.
Since $N(V)$ is open, this means that the support 
of $\mu$ does not intersect~$N(V)$.

\smallskip
We now prove  
\ref{thm:gpb-frob-ii}$\,\Rightarrow\,$\ref{thm:gpb-frob-i}.
The assumption $\Vhat(x)=V(x)$ for $|\mu|$-a.e.~$x$
and Theorem~\ref{thm:gpb-adv}\ref{thm:gpb-adv-ii} 
imply that $\Span(\tau'(x))\subset V(x)$ for $\mu'_a$-a.e.~$x$.  
On the other hand this inclusion holds also for $\mu'_s$-a.e.~$x$
by Theorem~\ref{thm:gpb-adv}\ref{thm:gpb-adv-i}, and therefore $T$ has
the geometric property of the boundary.
\end{proof}

\begin{proof}[Sketch of proof of Theorem~\ref{thm:ntv}]
Assume by contradiction that there exists $z\in A$
such that $u(z)\in N(V)$.
Since $N(V)$ is open and $u$ is continuous and  of class $W^{1,p}_\loc$
we can find a ball $U$ centered at $z$ such that 
\begin{itemize}
[leftmargin=25pt, itemsep=2pt]
\item
$u(U)$ is contained in $N(V)$;
\item
the restriction of $u$ to $\bd U$ belongs to $W^{1,p}(\bd U)$.
\end{itemize}
Then the graph of the restriction of $u$ to $U$, 
denoted by $\Gamma$, is a $k$-dimensional rectifiable
set with $\Haus^k(\Gamma)<+\infty$. 
Moreover it is proved in \cite{GMS}, \S2.5, Theorem~1, 
that the rectifiable current canonically associated
to $\Gamma$, denoted by $\ic{\Gamma}$, has boundary with 
finite mass, and therefore is an integral current 
in $\R^k\times\R^n$ (for this step we need the assumption~$p>k$).

Since $\nabla u$ has maximal rank, possibly replacing 
$U$ with a suitable open subset, we have that 
the pushforward of $\ic{\Gamma}$ through the projection 
$p:\R^k\times\R^n\to\R^n$ is a non-trivial integral
$k$-current tangent to $V$.
But the support of such current is contained in $N(V)$, 
thus violating Theorem~\ref{thm:frobenius}.
\end{proof}

    %
    %
    %
    %
\bibliography{bibliography}
\bibliographystyle{siam}

    %
    %
    %
    %
\vskip .5 cm
{\parindent = 0 pt\footnotesize
G.A.
\par
\smallskip
Dipartimento di Matematica,
Universit\`a di Pisa
\par
largo Pontecorvo 5,
56127 Pisa,
Italy
\par
\smallskip
e-mail: \texttt{giovanni.alberti@unipi.it}

\bigskip
A.M.
\par
\smallskip
Dipartimento di Tecnica e Gestione dei Sistemi Industriali (DTG), 
Universit\`a di Padova 
\par
stradella S.~Nicola 3, 
36100 Vicenza, 
Italy
\par
\smallskip
e-mail: \texttt{annalisa.massaccesi@unipd.it}

\bigskip
E.S.
\par
\smallskip
St.~Petersburg Branch of the Steklov Institute of Mathematics 
of the Russian Academy of Sciences
\par
Fontanka~27,
191023 St.~Petersburg, 
Russia
\par
\smallskip
and
\par
\smallskip
Faculty of Mathematics, Higher School of Economics, 
\par
Usacheva~6, 
119048 Moscow, 
Russia
\par
\smallskip
e-mail: \texttt{stepanov.eugene@gmail.com}
\par
}

\end{document}